\documentclass[12pt, reqno]{amsart}
\usepackage{amscd,amsmath,amsthm,amssymb,graphics}
\usepackage{amsfonts,amssymb,amscd,amsmath,enumitem,verbatim}
\usepackage[a4paper,top=3cm,left=3cm,right=3cm]{geometry}
\usepackage{todonotes}
\theoremstyle{plain}
\usepackage{orcidlink}
\usepackage{longtable}
\usepackage{color}
\usepackage{hyperref}
\newtheorem{Theorem}{Theorem}
\newtheorem{Lemma}[Theorem]{Lemma}
\newtheorem{Corollary}[Theorem]{Corollary}
\newtheorem{Proposition}[Theorem]{Proposition}

\newtheorem{Conjecture}[Theorem]{Conjecture}

\theoremstyle{definition}
\newtheorem{Definition}[Theorem]{Definition}

\newtheorem{Remark}[Theorem]{Remark}

\newtheorem{Example}[Theorem]{Example}

\usepackage{graphicx}
\usepackage{algorithm}
\usepackage{algpseudocode}
\usepackage{booktabs}
\usepackage{tabularx}
\usepackage{array}
\algrenewcommand\algorithmicrequire{\textbf{Input:}}
\algrenewcommand\algorithmicensure{\textbf{Output:}}

\newtheorem{innercustomgeneric}{\customgenericname}
\providecommand{\customgenericname}{}
\newcommand{\newcustomtheorem}[2]{%
  \newenvironment{#1}[1]
  {%
   \renewcommand\customgenericname{#2}%
   \renewcommand\theinnercustomgeneric{##1}%
   \innercustomgeneric
  }
  {\endinnercustomgeneric}
}

\newcustomtheorem{customthm}{Theorem}

\newcommand{\ord}{\text{ord}}

\title[]{Periodicity and Period-Length Bounds for Browkin $p$-Adic Continued Fractions}

\author[S. Barbero]{Stefano Barbero}
\address[S. Barbero]{Dipartimento di Matematica, Università di Trento, Via Sommarive 14, 38123, Povo (TN), Italy}
\email{stefano.barbero@unitn.it}

\author[N. Murru]{Nadir Murru}
\address[N. Murru]{Dipartimento di Matematica, Università di Trento, Via Sommarive 14, 38123, Povo (TN), Italy}
\email{nadir.murru@unitn.it}

\author[M. Urani]{Matilda Urani}
\address[M. Urani]{Dipartimento di Scienze Matematiche "Giuseppe Luigi Lagrange", Politecnico di Torino, Corso Duca degli Abruzzi 24, 10129, Torino (TO), Italy}
\email{matilda.urani@polito.it}

\thanks{}

\subjclass[2010]{}

\begin{document}

\begin{abstract}
Continued fractions have been introduced and studied in the field of $p$--adic numbers by several authors, with the aim of developing analogues of the classical theory of real continued fractions. In this paper, we focus on Browkin's continued fraction algorithm, for which several fundamental questions remain open, especially concerning periodic expansions of quadratic irrationals.
In this paper, we solve a conjecture about periodicity left open in \cite{CMT}. As a consequence, we prove that, for every positive integer $t$, there exist infinitely many square roots of integers whose Browkin continued fraction expansion is periodic with period length $2t$. This also settles a problem originating from previous classical results on the possible period lengths of square roots of integers. 
Moreover, we provide new sufficient conditions for the periodicity of quadratic irrationals and derive explicit bounds for the corresponding period lengths. These results also restrict the possible behaviour of a quadratic irrational whose Browkin continued fraction expansion is non-periodic, and therefore provide new tools for investigating whether an analogue of Lagrange's theorem can hold. 
Finally, we also provide a computational study of the Browkin expansions of quadratic irrationals, with particular attention to the behaviour of square roots of integers and to the search for possible non-periodic examples.
\end{abstract}

\maketitle

\section{Introduction}
Continued fractions over the field of $p$--adic numbers were first introduced by Mahler \cite{Mahler} with the aim of studying rational approximations and establishing $p$--adic counterparts of classical results in Diophantine approximation. However, explicit algorithms for computing the continued fraction expansion of $p$--adic numbers were deeply studied and developed only later. Among the best known and most extensively studied are those introduced by Schneider \cite{Sch}, Ruban \cite{Ruban}, and Browkin \cite{Bro1}.
Despite the extensive research devoted to these algorithms and their variants proposed over the years (such as the algorithms studied in \cite{Bro2, MR24, MRS23, Wang, Weg1, Yas}), there is still no canonical continued fraction algorithm in the $p$--adic setting that reproduces all the desirable properties of the classical ones, both in terms of rational approximations and periodic expansions. We refer the reader to \cite{Romeo1} for a complete survey on $p$--adic continued fractions.

In this field, one of the most important and challenging open problems is to characterize quadratic irrationals by means of periodic expansions, i.e., to find an algorithm for which an analogue of Lagrange's theorem holds. For both Schneider’s and Ruban’s algorithms, it is known that the analogue of Lagrange’s theorem fails. This was proved by de Weger \cite{Weg2} for Schneider’s algorithm and by Ooto \cite{Ooto} for that of Ruban, following the same argument developed by de Weger. In \cite{CVZ}, the authors also provided an effective criterion for determining when the expansion of a quadratic irrational is not periodic. By contrast, for Browkin’s algorithm, the problem of proving or disproving an analogue of Lagrange’s theorem appears to be considerably more difficult and remains unsolved.

Thus, there are many studies regarding the periodicity of the Browkin algorithm. For instance, Romeo \cite{Romeo2} studied the connection between  periodicity and real convergence of Browkin continued fractions.
Bedocchi \cite{Bed1, Bed2, Bed3} proved several results focused on the expansion of $\sqrt{D} \in \mathbb Q_p$ for $D$ a non--square integer. In particular, when $p$ is odd, he proved that if $\sqrt{D}$ has a periodic expansion then the preperiod has length 2. Moreover, he proved that $\sqrt{D}$ has never period 1 and, for any odd $t>1$, there are at most finitely many $\sqrt{D}$ with period of length $t$.
On the other hand, he found infinitely many square roots of integers with periods of length 2, 4, and 6. Pursuing this line of research, in \cite{CMT} the authors proved that, for every $s \geq 1$, there exist infinitely many square roots of integers whose Browkin continued fraction expansions are periodic with period length $2^s$. To this end, they introduced the notion of nice continued fractions and formulated some conjectures whose resolution would imply that, for every integer $t$, there exist infinitely many square roots of integers with a periodic Browkin expansion of period length $2t$, thereby settling a conjecture left open in Bedocchi’s work. 

In this paper, we prove these conjectures and we consequently solve the problem left open since Bedocchi’s work. Moreover, we also establish several sufficient conditions for periodicity, together with explicit bounds on the corresponding period lengths. These results are relevant to the broader problem of disproving an analogue of Lagrange’s theorem for Browkin continued fractions, as they provide new tools for the search for possible counterexamples. More precisely, the sufficient conditions and period-length bounds substantially restrict the range of cases in which such a counterexample may occur. We also investigate these situations through computational considerations and numerical examples, which further clarify where a failure of periodicity might be expected.

The paper is structured as follows. In Section \ref{sec:pre}, we fix the notation and recall some preliminary facts and results about Browkin continued fractions. In Section \ref{sec:nice}, we solve the conjectures left open in \cite{CMT}. Section \ref{sec:bound} is devoted to the study of sufficient conditions for periodicity, improving and expanding some results obtained in \cite{CMT}. Finally, in Section \ref{sec:comp}, we present some computations about nice continued fractions and the expansions of quadratic irrationals. 

\section{Preliminaries and notation}\label{sec:pre}

In the following, we will fix an odd prime $p$ and we denote by $v_p(\cdot)$ the $p$--adic valuation, $|\cdot|_p$ the $p$--adic norm and $|\cdot|$ the Euclidean norm. A $p$--adic number $\alpha$ will be represented by a formal power series $\alpha = \sum_{n=r}^\infty b_n p^n$, for $r \in \mathbb Z$ and the digits $b_n$'s are taken in the set of centered representatives modulo $p$.

We define the function $s: \mathbb Q_p \rightarrow \mathbb Q$ by 
\[ \begin{cases} s(\alpha) = \sum_{n=r}^0 b_n p^n, \quad \text{if $r \leq 0$} \cr 0, \quad \text{otherwise} \end{cases}. \]
More precisely, the image of $s$ is given by the set $\mathbb Z[\frac{1}{p}] \cap \left(-\frac{p}{2}, \frac{p}{2}\right)$ which is discrete in $\mathbb Q_p$ and plays the same role as the real floor function \cite{Bro1, Lager}.

Then, given $\alpha_0 \in \mathbb Q_p$, its Browkin continued fraction expansion is computed by 
\[\begin{cases} a_n = s(\alpha_n) \cr \alpha_{n+1} = \cfrac{1}{\alpha_n-a_n} \end{cases} \]
for all $n=0,1, \ldots$ and we write $\alpha_0 = [a_0, a_1, \ldots]$, where $a_n$'s and $\alpha_n$'s are called partial quotients and complete quotients, respectively. If $\alpha_n = s(\alpha_n)$ for some index $n$, then the continued fraction is finite of length $n+1$. 

Given a Browkin continued fraction $[a_0, a_1, \ldots]$, the convergents are defined in the usual way by
\[ \cfrac{A_n}{B_n} := [a_0, \ldots, a_n]\]
for all positive $n$ less than the length of the continued fraction. Moreover, it is also convenient to define 
\[ A_{-2} := 0, \quad A_{-1} := 1, \quad B_{-2} := 1, \quad B_{-1} := 0 \]
so that numerators and denominators of the convergents can be computed by the usual recurrences
\[ A_n = a_n A_{n-1} + A_{n-2}, \quad B_n = a_n B_{n-1} + B_{n-2} \]
for $n = 0, 1, \ldots$. Observe that, in the case of Browkin continued fractions, $A_n$'s and $B_n$'s are rational numbers but they are also commonly referred to as numerators and denominators of convergents. 

In the next proposition, we recall some basic properties (see, e.g., \cite{Bro1}).

\begin{Proposition}
Given $\alpha_0 = [a_0, a_1, \ldots]$, then
\begin{enumerate}
\item $v_p(a_n) = v_p(\alpha_n) < 0$ and $|a_n|_p = |\alpha_n|_p > 1$, for all $n \geq 1$;
\item $v_p(B_n) = v_p(a_1) + \ldots + v_p(a_n)$ and $|B_n|_p = |a_1\cdots a_n|_p$, for all $n \geq1$;
\item if $a_0 \not= 0$, then $v_p(A_n) = v_p(a_0) + \ldots + v_p(a_n)$ and $|A_n|_p = |a_0 \cdots a_n|_p$ for all $n \geq 0$;
\item if $a_0 = 0$, then $v_p(A_n) = v_p(a_2) + \ldots + v_p(a_n)$ and $|A_n|_p = |a_2 \cdots a_n|_p$ for all $n \geq 0$;
\item $v_p\left( \alpha_0 - \frac{A_n}{B_n} \right) = -v_p(B_n B_{n+1})$ for all $n \geq 0$
\end{enumerate}
\end{Proposition}

Since $a_n, A_n, B_n$ are all rational numbers in $\mathbb Z[\frac{1}{p}]$, we also introduce the sequences of integers $(\tilde a_n)_{n \geq0}, (\tilde A_n)_{n \geq 0}, (\tilde B_n)_{n \geq 0}$ such that
\begin{equation}\label{eq:tilde}
a_n = \cfrac{\tilde a_n}{p^{k_n}}, \quad A_n = \cfrac{\tilde A_n}{p^{\sum_{i=\varepsilon}^n k_i}}, \quad B_n = \cfrac{\tilde B_n}{p^{\sum_{i=1}^n k_i}} 
\end{equation}
where $p \nmid \tilde a_n \tilde A_n \tilde B_n$, $k_n = -v_p(a_n)$, and $\varepsilon = 0 \text{ or } 2$ according to the previous proposition. Given these definitions, it is straightforward to check that
\begin{equation}\label{eq:rel_ABTilde}
    \Tilde{A}_{n} = \Tilde{a}_n\Tilde{A}_{n-1} + p^{k_n+k_{n-1}}\Tilde{A}_{n-2} , \quad
    \Tilde{B}_{n} = \Tilde{a}_n\Tilde{B}_{n-1} + p^{k_n+k_{n-1}}\Tilde{B}_{n-2} ,
\end{equation}
for all $n \geq 2$, from which it also follows that
\begin{equation} \label{eq:gcd}
\gcd(\tilde A_n, \tilde A_{n-1}) = 1, \quad \gcd(\tilde B_n, \tilde B_{n-1}) = 1.
\end{equation}

We now recall the definition of nice continued fractions introduced in \cite{CMT}, which enabled the authors to prove that, for every $s \geq 1$, there exist infinitely many square roots of integers whose Browkin  continued fraction expansion is periodic with period length $2^s$.

\begin{Definition}[\cite{CMT}, Definition 5.2]\label{def:niceCF}
A finite Browkin continued fraction $[a_0,\dots, a_{t-1}]$ is \textit{nice} if 
\begin{itemize}
 \item[a.] $|a_0|_p > 1$ and $|a_0| < \frac{p}{4}$;
 \item[b.] $\left|\frac{A_{t-1}}{A_{t-2}}\right| > \frac{4}{p}$;
 \item[c.] there exists an integer $q$ such that $\Tilde{B}_{t-1} \mid q \mid \Tilde{B}_{t-1}^2$ and the class of $q$ modulo $\Tilde{A}_{t-1}^2$ belongs to the multiplicative subgroup generated by the class of $p$. \end{itemize}
\end{Definition}

The following theorem is the main tool developed in \cite{CMT} that allowed the authors to prove their main results.

\begin{Theorem}[\cite{CMT}, Theorem 5.6]
Let $[a_0, \ldots, a_{t-1}]$ be a nice continued fraction, then there exist infinitely many $a_t$ such that $[a_0, \overline{a_1, \ldots , a_{t-1}, a_t, a_{t-1}, \ldots, a_1, 2a_0}]$ converges to a quadratic irrational $\frac{1}{p^{k_0}\sqrt{m}}$ for some $m \in \mathbb Z$.
\end{Theorem}

From the periodic expansion of a quadratic irrational $\frac{1}{p^{k_0}\sqrt{m}}$ as given in the previous theorem, it is possible to derive the periodic expansion of its inverse $p^{k_0}\sqrt{m} = [0, a_0, \overline{a_1, \ldots , a_{t-1}, a_t, a_{t-1}, \ldots, a_1, 2a_0}]$. Thus, proving that there exist infinitely many square roots of integers with periodic Browkin continued fraction expansions of every prescribed even length reduces to showing that, for every $t \geq 1$, there exists a nice continued fraction of length $t$. For this reason, the following conjectures were proposed.

\begin{Conjecture}[\cite{CMT}, Conjectures 5.7, 5.8, 5.9] \label{conj} \
\begin{enumerate}
\item For every $t \geq 1$ there exists a nice continued fraction of length $t$ (except when $t = 1$ and $p = 3$).
\item For every partial quotients $a_0, \ldots, a_{t-2} > 0$, with $|a_0| <p/4$ and $|a_0|_p>1$, there exists $a_{t-1}$ such that the Browkin continued fraction $[a_0, \ldots, a_{t-1}]$ is nice.
\item For every $t \geq 1$ there exists $a_{t-1}$ such that the Browkin continued fraction $[1/p, \ldots, 1/p, a_{t-1}]$ is nice.
\end{enumerate}
\end{Conjecture}

In the next section, we are able to prove the second point of the previous conjecture (and hence all the other ones) even without assuming that the partial quotients are positive. Consequently, we also prove that, for every integer $t$, there exist infinitely many square roots of integers with a periodic Browkin expansion of period length $2t$ (which also corresponds to Conjecture 5.10 of \cite{CMT} and to the problem left open since Bedocchi's works). Since the conjecture has already been settled for $t=1$, in what follows we will focus on the case $t>1$.

\section{Nice sequences of arbitrary length}\label{sec:nice}
Let us consider a sequence of partial quotients $a_0, \ldots, a_{t-2}$ such that $|a_0| < p/4$ and $|a_0|_p >1$. We will show that there exists $a_{t-1}  = \frac{\tilde a_{t-1}}{p^{k_{t-1}}} \in \mathbb Z[\frac{1}{p}] \cap \left(-\frac{p}{2},\frac{p}{2}\right)$ such that 
\begin{equation}\label{eq:fc}
[a_0, \ldots, a_{t-1}]
\end{equation}
is nice, proving Conjecture \ref{conj}.

Observe that the sequences $(\Tilde{A}_n)_{n\geq0}$ and $(\Tilde{B}_n)_{n\geq0}$ of the continued fraction \eqref{eq:fc}, for $n\le t-2$, only depend on the given partial quotients $a_0,\dots,a_{t-2}$. On the other hand, for $n=t-1$, we have 
\begin{equation} \label{eq:ric-tilde}
\Tilde{A}_{t-1} = \Tilde{a}_{t-1}\Tilde{A}_{t-2}
           + p^{\kappa}\Tilde{A}_{t-3}, \quad
\Tilde{B}_{t-1} = \Tilde{a}_{t-1}\Tilde{B}_{t-2}
           + p^{\kappa}\Tilde{B}_{t-3},
\end{equation}
where $\kappa := k_{t-1}+k_{t-2}$, and $k_{t-1},\Tilde{a}_{t-1}$ are the values that we want to find so that \eqref{eq:fc} is nice.
\begin{Lemma} \label{lemma:rel_sign2_5.8}
For the continued fraction \eqref{eq:fc}, it holds that 
\begin{equation*}
    \Tilde{A}_{t-2}\Tilde{B}_{t-1}
    - \Tilde{B}_{t-2}\Tilde{A}_{t-1}
    = (-1)^{t-1}\,p^{\kappa + \sigma}\,,
\end{equation*}
where $\sigma := \sum_{i=0}^{t-3}(k_i+k_{i+1})$.
\end{Lemma}
\begin{proof}
The classical fundamental relation for numerators and denominators of convergents also holds for Browkin continued fractions:
\begin{equation}\label{eq:fund}
    A_{t-2}B_{t-1}-B_{t-2}A_{t-1} = (-1)^{t-1}.
\end{equation}
Moreover, according to \eqref{eq:tilde}, it holds
\begin{equation}\label{eq:rel_betweenA_ATilde}
    \Tilde{A}_n=p^{\sum_{i=0}^nk_i}{A}_n, \qquad \Tilde{B}_n=p^{\sum_{i=1}^nk_i}{B}_n.
\end{equation}
Hence, we can write
\begin{equation*}
    \frac{\Tilde{A}_{t-2}\Tilde{B}_{t-1}}{p^{\sum_{i=0}^{t-2}k_i}p^{\sum_{i=1}^{t-1}k_i}}-\frac{\Tilde{B}_{t-2}\Tilde{A}_{t-1}}{p^{\sum_{i=1}^{t-2}k_i}p^{\sum_{i=0}^{t-1}k_i}} = (-1)^{t-1}.
\end{equation*}
That is, \begin{equation*}
    {\Tilde{A}_{t-2}\Tilde{B}_{t-1}}-{\Tilde{B}_{t-2}\Tilde{A}_{t-1}} = (-1)^{t-1}{p^{{\sum_{i=0}^{t-2}k_i+k_{i+1}}}}.
\end{equation*}
\end{proof}

\begin{Lemma}\label{lemma:powers_p}
    Let $p,r$ be two distinct odd primes. Let $h$ be the order of $p$ modulo $r$ and $w = {v}_r(p^h-1)$. Then, for every non-zero integer $c$ and $e \ge w$, the element  $1+cr^e$ is a power of $p$ modulo $r^{2e}$.
\end{Lemma}
\begin{proof}
Let us consider $g:=(p^h)^{r^{e-w}} $.
Then, it holds that 
\begin{equation*}
v_r(g-1) = v_r((p^h)^{r^{e-w}}-1) = v_r(p^h-1) + v_r({r^{e-w}}) = w+e-w= e,
\end{equation*}
where the second equality follows by the lifting-the-exponent lemma. Then, 
\begin{equation*}
g \equiv 1+ u r^e \mod r^{2e},
\end{equation*}
for a certain $u$, with  $r\nmid u$.
Let $m$ be an integer such that $m \cdot u \equiv c \mod r^{2e}$, which always exists since $u$ has an inverse modulo $r^{2e}$. Then, \begin{equation*}
g^m \equiv (1+ur^e)^m \equiv  1+mur^e \equiv 1+cr^e \mod r^{2e},
\end{equation*}
where the second equivalence follows directly from the binomial expansion. 
Since $g$ is a power of $p$, this shows that also $1+cr^e$ is a power of $p$ modulo $r^{2e}$.
\end{proof}

\begin{Theorem}\label{thm:C}
There exist infinitely many $\Tilde{a}_{t-1}$ and $k_{t-1}$ such that 
\begin{itemize}
\item[(i)]  $\Tilde{A}_{t-2} \mid \Tilde{B}_{t-1}$ 
\item[(ii)] $\Tilde{A}_{t-1}= \pm r^e$ 
\end{itemize}
for a certain prime $r\not=p$ and some integer $e \equiv 1 \pmod {\lambda(\tilde A_{t-2}^2)}$, where $\lambda$ is the Carmichael function. 
\end{Theorem}
\begin{proof}
We start by observing that 
\begin{align*}
    \Tilde{A}_{t-2} \mid \Tilde{B}_{t-1} &\Leftrightarrow  \Tilde{B}_{t-1} \equiv 0 \mod \Tilde{A}_{t-2}  \\ &\Leftrightarrow \Tilde{a}_{t-1}\Tilde{B}_{t-2} + p^{\kappa}\Tilde{B}_{t-3}\equiv 0 \mod \Tilde{A}_{t-2}\\ &\Leftrightarrow   \Tilde{a}_{t-1}  \equiv -p^{\kappa}\Tilde{B}_{t-3}\Tilde{B}_{t-2}^{-1} \mod \Tilde{A}_{t-2} \,\,,
\end{align*}
where the last equivalence follows from $\gcd(\Tilde{B}_{t-2},\Tilde{A}_{t-2})=1$ (which can be derived from \eqref{eq:gcd} and \eqref{eq:fund}). Thus, there are infinitely many $\Tilde a_{t-1}$ such that $\Tilde A_{t-2} \mid \Tilde B_{t-1}$.
Now, we want to show that there exist $r,e$ (as in the hypothesis of the Theorem) such that $\Tilde{A}_{t-1}= \varepsilon r^e$, for $\varepsilon =\pm1$. Since we have \begin{align*}
    \Tilde{A}_{t-1} &= \Tilde{a}_{t-1}\Tilde{A}_{t-2} + p^{\kappa}\Tilde{A}_{t-3} \\  &\equiv (-p^{\kappa}\Tilde{B}_{t-3}\Tilde{B}_{t-2}^{-1})\Tilde{A}_{t-2} + p^{\kappa}\Tilde{A}_{t-3} \mod \Tilde{A}_{t-2}^2 \\
    & \equiv p^{\kappa}(\Tilde{A}_{t-3}-\Tilde{B}_{t-3}\Tilde{B}_{t-2}^{-1}\Tilde{A}_{t-2})  \mod \Tilde{A}_{t-2}^2,
\end{align*}
then to satisfy $(ii)$, we want \begin{equation}\label{eq:r_condition_5.8}
\varepsilon r^e \equiv p^{\kappa}(\Tilde{A}_{t-3}-\Tilde{B}_{t-3}\Tilde{B}_{t-2}^{-1}\Tilde{A}_{t-2})  \mod \Tilde{A}_{t-2}^2.  
\end{equation}
It is easy to verify that such $r$ exists if we require that \begin{itemize}
    \item $e\equiv 1 \mod \lambda(\Tilde{A}_{t-2}^2)$, where $ \lambda$ is the Carmichael function;
    \item $\kappa \equiv 1 \mod \ord_{\Tilde{A}_{t-2}^2}(p)$; 
\end{itemize}
According to these conditions, \eqref{eq:r_condition_5.8} becomes \begin{equation*}
    r \equiv \varepsilon p(\Tilde{A}_{t-3}-\Tilde{B}_{t-3}\Tilde{B}_{t-2}^{-1}\Tilde{A}_{t-2}) \mod \Tilde{A}_{t-2}^2,    
\end{equation*} 
and by Dirichlet's theorem on arithmetic progressions there exist infinitely many primes $r$ satisfying the above congruence provided that $$\gcd(\Tilde{A}_{t-3}-\Tilde{B}_{t-3}\Tilde{B}_{t-2}^{-1}\Tilde{A}_{t-2},\Tilde{A}_{t-2}^2)=1.$$ 
Assume by contradiction that $\ell$ is a common prime divisor. 
Hence, $\ell \mid \Tilde{A}_{t-2}$ and from $\ell \mid \Tilde{A}_{t-3}-\Tilde{B}_{t-3}\Tilde{B}_{t-2}^{-1}\Tilde{A}_{t-2}$, we also get $\ell \mid \Tilde{A}_{t-3}$, which is not possible since $\gcd(\Tilde{A}_{t-2},\Tilde{A}_{t-3})=1$, thereby concluding the proof.
\end{proof}

Exploiting the previous theorem, we are able to find infinitely many $a_{t-1}$ such that condition $(c)$ of Definition \ref{def:niceCF} is fulfilled. Indeed, once $a_{t-1}$ is selected in order to satisfy conditions $(i)$ and $(ii)$ of Theorem \ref{thm:C}, it is sufficient to consider $q$ as
\begin{equation}\label{eq:q5.8}
    q = (-1)^{t-1}\Tilde{A}_{t-2}\Tilde{B}_{t-1}.
\end{equation}
Now, it is easy to verify that $\Tilde{B}_{t-1} \mid q \mid \Tilde{B}_{t-1}^2$. 

By Lemma \ref{lemma:rel_sign2_5.8}, we have
\begin{align*}
    q &= (-1)^{t-1}\Tilde{A}_{t-2}\Tilde{B}_{t-1} = \\ &=(-1)^{t-1}((-1)^{t-1}p^{\kappa + \sigma}+ \Tilde{B}_{t-2}\Tilde{A}_{t-1})  = \\&=p^{\kappa + \sigma} (1+ (-1)^{t-1}p^{-\kappa -\sigma}\Tilde{B}_{t-2}\Tilde{A}_{t-1}).
\end{align*}

Moreover, observe that we can always choose $e$ in Theorem \ref{thm:C} sufficiently large so that
\[
e \geq v_r(p^{\operatorname{ord}_r(p)}-1).
\]
Thus, according to Lemma \ref{lemma:powers_p}, it follows that $(1+ (-1)^{t-1}p^{-\kappa  -\sigma}\Tilde{B}_{t-2}\Tilde{A}_{t-1})$ is a power of $p$ modulo $ \Tilde{A}_{t-1}^2$, i.e., $q$ belongs to the multiplicative group generated by $p$ modulo $\Tilde{A}_{t-1}^2$.

We still need to prove that among the infinitely many values of $a_{t-1}$ provided by Theorem \ref{thm:C}, there exist $\tilde a_{t-1}$ and $k_{t-1}$ such that also condition $(b)$ of Definition \ref{def:niceCF} is satisfied and $a_{t-1}$ is an allowed partial quotient for Browkin continued fractions, i.e., $|a_{t-1}| < p/2$.

\begin{Theorem} \label{thm:B}
There exists $a_{t-1} = \frac{\tilde a_{t-1}}{p^{k_{t-1}}}$ satisfying conditions $(i)$, $(ii)$ of Theorem \ref{thm:C} such that
\[ \left| \cfrac{A_{t-1}}{A_{t-2}}\right| > \cfrac{4}{p}, \quad |a_{t-1}| < \cfrac{p}{2}. \]
\end{Theorem}
\begin{proof}
From the proof of Theorem \ref{thm:C} and by \eqref{eq:ric-tilde}, we derive that the values $\Tilde{a}_{t-1}$ and $k_{t-1}$ are selected such that \begin{align*}
    \Tilde{a}_{t-1} = \frac{\varepsilon r^e-p^{\kappa}\Tilde{A}_{t-3}}{\Tilde{A}_{t-2}},
\end{align*}
where \begin{itemize}
    \item[-] $\varepsilon = \pm 1$ and $r \not=p$ prime;
    \item[-] $e\equiv 1 \mod \lambda(\Tilde{A}_{t-2}^2)$;
    \item[-] $\kappa \equiv 1 \mod \ord_{\Tilde{A}_{t-2}^2}(p)$, that is $k_{t-1} \equiv 1 - k_{t-2}\mod \ord_{\Tilde{A}_{t-2}^2}(p)$.
\end{itemize}
By equation \eqref{eq:rel_betweenA_ATilde}, we have
\[
   \frac{A_{t-1}}{A_{t-2}}=\frac{\Tilde A_{t-1}}{ p^{k_{t-1}}\Tilde A_{t-2}} = \frac{\varepsilon\,{r^{e}}}{p^{k_{t-1}}\Tilde A_{t-2}}.
\]
Condition $(b)$ of Definition \ref{def:niceCF}, namely $\bigl|\frac{A_{t-1}}{A_{t-2}}\bigr|>\tfrac4p$, becomes
\begin{equation}\label{eq:condition_b_r}
      \frac{r^{e}}{p^{k_{t-1}}}>\tfrac4p\,\bigl|\Tilde A_{t-2}\bigr| ,
\end{equation}
independently of $\varepsilon$. 
Moreover, to have a valid partial quotient, we also require that \begin{equation} \label{eq:cond-bro}
    | a_{t-1} |<\frac{p}{2},
\end{equation}
which can be written as \begin{align*}
    \left| \frac{\varepsilon r^e-p^\kappa\Tilde{A}_{t-3}}{p^{k_{t-1}}\Tilde{A}_{t-2}} \right|<\frac{p}{2} &\quad \Leftrightarrow \quad  {\left|\frac{\varepsilon r^e}{p^{k_{t-1}}}-p^{k_{t-2}}\Tilde{A}_{t-3}\right|} <\frac{p}{2}{\left|\Tilde{A}_{t-2}\right|}.
\end{align*}
Choose now $\varepsilon:=\operatorname{sign}(\Tilde A_{t-3})$, then \begin{align*}
    \varepsilon p^{k_{t-2}}\Tilde A_{t-3} &=p^{k_{t-2}}\bigl|\Tilde A_{t-3}\bigr| \\
    \bigl|\varepsilon \frac{r^{e}}{p^{k_{t-1}}}-p^{k_{t-2}}\Tilde A_{t-3}\bigr|&=\bigl|\,\frac{r^{e}}{p^{k_{t-1}}}-p^{k_{t-2}}|\Tilde A_{t-3}|\,\bigr|,
\end{align*} 
and condition \eqref{eq:cond-bro} becomes
\begin{equation}\label{eq:admissibility}
   p^{k_{t-2}}|\Tilde A_{t-3}|-\tfrac p2\bigl|\Tilde A_{t-2}\bigr|<\frac{r^{e}}{p^{k_{t-1}}}<p^{k_{t-2}}|\Tilde A_{t-3}|+\tfrac p2\bigl|\Tilde A_{t-2}\bigr| .
\end{equation}
Hence, equations \eqref{eq:admissibility} and \eqref{eq:condition_b_r} give \begin{equation*}
    \max\left(\tfrac4p|\Tilde{A}_{t-2}|\,,\,p^{k_{t-2}}|\Tilde A_{t-3}|-\tfrac p2\bigl|\Tilde A_{t-2}\bigr|\right) <\frac{r^{e}}{p^{k_{t-1}}}<  p^{k_{t-2}}|\Tilde A_{t-3}|+\tfrac p2\bigl|\Tilde A_{t-2}\bigr| .
\end{equation*}
By denoting with $\eta:= \max\left(\tfrac4p|\Tilde{A}_{t-2}|\,,\,p^{k_{t-2}}|\Tilde A_{t-3}|-\tfrac p2\bigl|\Tilde A_{t-2}\bigr|\right) $ and $\mu :=p^{k_{t-2}}|\Tilde A_{t-3}|+\tfrac p2\bigl|\Tilde A_{t-2}\bigr|$, we can write \begin{equation}\label{eq:etamu}
    \log_p(\eta) < e\log_p(r)-k_{t-1} <\log_p\left(\mu\right).
\end{equation}
Thus, the last thing to prove is that there exist such $e$ and $k_{t-1}$ satisfying the above inequalities. We recall that 
$$e = 1 + n_1 \lambda(\Tilde{A}_{t-2}^2), \quad k_{t-1}= 1- k_{t-2}+ n_2\ord_{\Tilde{A}_{t-2}^2}(p)$$ 
for some integers $n_1, n_2$. 
Hence, we can rewrite \eqref{eq:etamu} as  
\begin{equation}
    \eta' <  n_1 \lambda(\Tilde{A}_{t-2}^2)\log_p(r)- n_2\ord_{\Tilde{A}_{t-2}^2}(p) <\mu'.
\end{equation}
where $\eta' =\log_p(\eta) - \log_p(r) + 1- k_{t-2}$ and  $\mu' = \log_p(\mu) - \log_p(r) + 1- k_{t-2}$.
 We finally need to prove that
 \begin{itemize}
     \item[1.] the interval $(\eta', \mu') \neq \emptyset$;
     \item[2.] there exist two integers $n_1,n_2$ such that $n_1 \lambda(\Tilde{A}_{t-2}^2)\log_p(r)-n_2\ord_{\Tilde{A}_{t-2}^2}(p)$ is contained in the interval $(\eta', \mu')$;
 \end{itemize}
The first assumption is equivalent to showing that the interval $(\,\log_p(\eta), \log_p(\mu)\,)$ is non-empty. Since the logarithm is strictly increasing, it is sufficient to prove that $\eta < \mu$. Since $p \ge 3$, both $\tfrac4p|\Tilde{A}_{t-2}|\,,\,p^{k_{t-2}}|\Tilde A_{t-3}|-\tfrac p2\bigl|\Tilde A_{t-2}\bigr|$ are strictly smaller than $p^{k_{t-2}}|\Tilde A_{t-3}|+\tfrac p2\bigl|\Tilde A_{t-2}|$, hence $\eta < \mu$.

For the second assumption, we start by observing that 
\begin{itemize}
    \item $\log_p(r) \notin \mathbb{Q}$ (otherwise $r=p^{a/b}$ for some $a,b$, which is impossible for two distinct primes);
    \item $\frac{\lambda(\Tilde{A}_{t-2}^2)\log_p(r)}{\ord_{\Tilde{A}_{t-2}^2}(p)}$ is irrational (since $\log_p(r)$ is irrational and $\lambda(\Tilde{A}_{t-2}^2),\ord_{\Tilde{A}_{t-2}^2}(p)$ are integers);
\end{itemize}
Moreover, in general, given an irrational number $\alpha$, the sequence  $(n\alpha \mod 1)_{n\geq0}$ is equidistributed. This also implies that the set $\{n\alpha \mod 1 \mid n \in \mathbb{N} \}$ is dense in $[0,1)$.
From these observations, we can conclude that \begin{equation*}
    \{n_1 \lambda(\Tilde{A}_{t-2}^2)\log_p(r) \mod \ord_{\Tilde{A}_{t-2}^2}(p) \mid n_1 \in \mathbb{N} \}
\end{equation*}
is dense in $[0,\ord_{\Tilde{A}_{t-2}^2}(p))$. Hence, 
 \begin{equation*}
    \{n_1 \lambda(\Tilde{A}_{t-2}^2)\log_p(r) -n_2\ord_{\Tilde{A}_{t-2}^2}(p) \mid n_1 ,n_2\in \mathbb{N} \}
\end{equation*}
is dense in $\mathbb{R}$. This means that for every non-empty open interval of $\mathbb{R}$ there exist infinitely many values of the form $n_1 \lambda(\Tilde{A}_{t-2}^2)\log_p(r) -n_2\ord_{\Tilde{A}_{t-2}^2}(p)$. Thus, since $(\eta', \mu') \neq \emptyset$, our argument is proved.
\end{proof}

Hence, Theorems \ref{thm:C} and \ref{thm:B} show that we are always able to find $a_{t-1} \in \mathbb Z[\frac{1}{p}] \cap \left(-\frac{p}{2}, \frac{p}{2}\right)$ such that $[a_0, \ldots, a_{t-2}, a_{t-1}]$ is a nice continued fraction of length $t$, provided that in the starting sequence $a_0, \ldots, a_{t-2}$, the first partial quotient satisfies condition $(a)$ of Definition \ref{def:niceCF}. Moreover, these theorems provide a constructive method to determine such a value, although it is far from being efficient, as we will discuss in Section \ref{sec:comp-nice}. 

\section{Some bounds on the period length}\label{sec:bound}

In \cite{CMT}, the authors established sufficient conditions for the periodicity of Browkin continued fractions in terms of norms of the complete quotients and, under these conditions, also obtained a bound on the period length. We recall these results below.

\begin{Theorem} \label{thm:CMT}
Let $\alpha \in \mathbb Q_p$ be a quadratic irrational whose discriminant $\Delta > 0$ is a non-square integer and for every $n > 0$ denote by $\xi_n$ and $\xi_n'$ the two images of $\alpha_n$ in $\mathbb C$. Assume that $\exists n_0 > 0$ such that
\begin{enumerate}
\item $N(\xi_n) < 0$ for every $n \in [n_0, n_0 + K]$ with 
\begin{equation} \label{eq:bound-CMT}
K := (2d + 1) \Delta + 1 - \frac{d(d+1)(2d+1)}{3},
\end{equation}
where $d = \lfloor \sqrt{\Delta} \rfloor$, then the Browkin continued fraction expansion of $\alpha$ is periodic with period of length at most $K$ (\cite[Proposition 4.8]{CMT});
\item $N(\xi_n)$ and $N(\xi_{n+1})$ have alternating signs for every $n \in [n_0, n_0 + 2K]$, then the Browkin continued fraction expansion of $\alpha$ is periodic of period of length at most $2K$ (\cite[Proposition 4.11]{CMT}).
\end{enumerate}
\end{Theorem}

In this section, first of all, we improve the bounds provided in the previous theorem and then generalize the result to the case where, within a prescribed interval, a given number of complete quotients have negative norm. We then address the more difficult case in which the complete quotients eventually all have positive norm. This case is harder because positivity does not immediately provide the boundedness needed to restrict the complete quotients to a finite set. This also distinguishes Browkin continued fractions from Ruban ones, for which periodic expansions can not have eventually positive norms, whereas periodic Browkin expansions with all complete quotients of positive norm do exist. 
Indeed, it is easy to exhibit periodic Browkin continued fractions of arbitrary even period length whose complete quotients all have positive norm, as in the following example.
\begin{Example}
Given any $L \in \mathbb N$ even, consider $y_i \in \mathbb Z[\frac{1}{p}] \cap \left(-\frac{p}{2}, \frac{p}{2}\right)$, with $y_i > 2$, for every $0 \leq i \leq L-1$ and construct the purely periodic continued fraction $$\alpha_0 = [\overline{a_0, \ldots, a_{L-1}}],$$ where $a_i := (-1)^i y_i$. We can observe that the complete quotients are
\[ \alpha_n = [\overline{a_n, a_{n+1}, \ldots, a_{L-1}, a_0, a_1, \ldots, a_{n-1}}], \]
for all $1 \leq n \leq L-1$. A quadratic polynomial for which these quadratic irrationals are roots is
\[ B^{(n)}_{L-1} x^2 + (B^{(n)}_{L-2} - A^{(n)}_{L-1}) x - A^{(n)}_{L-2}, \]
where $(A^{(n)}_i)_{i \geq 0}$ and $(B^{(n)}_i)_{i \geq 0}$ are the sequences of numerators and denominators of convergents of $\alpha_n$ for $0 \leq n \leq L-1$. Thus, we have that $N(\xi_n) = -\frac{A_{L-2}^{(n)}}{B_{L-1}^{(n)}}$, where $\xi_n$ is the image of $\alpha_n$ in $\mathbb C$ and
\[ \text{sign}(A^{(n)}_{L-2}) = \text{sign}(a_0 \cdots a_{L-2}), \quad  \text{sign}(B^{(n)}_{L-1}) = \text{sign}(a_1 \cdots a_{L-1}). \]

For instance, let us prove this for \(A_{L-2}^{(0)}\).
We show by induction that, for every \(1 \leq i \leq L-2\),
\[
\operatorname{sign}\left(A_i^{(0)}\right)
= \operatorname{sign}(a_0\cdots a_i), \quad \left|A_i^{(0)}\right|>\left|A_{i-1}^{(0)}\right|.
\]
For $i = 0$, it is straightforward to check and for \(i=1\), since \(a_0a_1<0\) and \(|a_0a_1|>4\), we have that \(A_1^{(0)}=a_0a_1+1\) has the sign of \(a_0a_1\), and moreover
\[
\left|A_1^{(0)}\right|= |a_0a_1|-1 > |a_0| = \left|A_0^{(0)}\right|.
\]
Assume now that the claim holds up to \(i-1\).
Since \(a_{i-1}\) and \(a_i\) have opposite signs,
\(a_iA_{i-1}^{(0)}\) and \(A_{i-2}^{(0)}\) have opposite
signs. Moreover, using \(|a_i|>2\),
\[
|a_iA_{i-1}^{(0)}|>2|A_{i-1}^{(0)}|>|A_{i-1}^{(0)}|+|A_{i-2}^{(0)}|.
\]
Therefore, from $A_i^{(0)}=a_iA_{i-1}^{(0)}+A_{i-2}^{(0)},$
we obtain
\[
\operatorname{sign}\left(A_i^{(0)}\right)=
\operatorname{sign}\left(a_iA_{i-1}^{(0)}\right)=
\operatorname{sign}(a_0\cdots a_i)
\]
and $|A_i^{(0)}|>|A_{i-1}^{(0)}|.$ Thus, in particular,
$\operatorname{sign}\left(A_{L-2}^{(0)}\right)=
\operatorname{sign}(a_0\cdots a_{L-2}).$

Considering that $a_i < 0$ if and only if $i$ is odd, we have
\[ \text{sign}(A^{(n)}_{L-2}) = \text{sign}\left(\prod_{i=1}^{L-3} a_i\right), \quad \text{sign}(B^{(n)}_{L-1}) = \text{sign}\left(\prod_{i=1}^{L-1} a_i\right),\]
i.e., $A^{(n)}_{L-2}$ and $B^{(n)}_{L-1}$ have opposite sign and consequently $N(\xi_n) > 0$ for every $0 \leq n \leq L-1$. 
\end{Example}

Let us introduce some notation used in the following. Given $\alpha \in \mathbb{Q}_p$ quadratic irrational, we write \begin{equation*}
    \alpha = \frac{b_0+\delta}{p^{k_0}c_0},
\end{equation*}
where $b_0,c_0,k_0 \in \mathbb{Z}, p \nmid c_0, \delta \in \mathbb Q_p$ and let $\Delta=\delta^2$ be a non-square positive integer. Moreover, for any complete quotient $\alpha_n$ we have
\begin{equation*}
    \alpha_n = \frac{b_n+\delta}{p^{k_n}c_n},
\end{equation*}
where $b_n,c_n,k_n \in \mathbb{Z}, p \nmid c_n$ and \begin{equation}\label{eq:bncnrecurrence}
    \begin{cases}
    b_n+b_{n+1} = a_np^{k_n}c_n \\
    p^{k_n+k_{n-1}}c_nc_{n-1}= \Delta -b_{n}^2
\end{cases}
\end{equation}
We will also denote by $\xi_n$ the image of $\alpha_n$ in $\mathbb C$ and define $d := \lfloor \sqrt{\Delta} \rfloor, C_n := p^{k_n}c_n$, $D_n := |C_n|$, for every $n \geq 0$. Finally, if $\alpha$ has a periodic expansion, we will denote by $\lambda(\alpha)$ the length of the period.

\begin{Proposition} \label{prop:improv}
Let $\alpha\in\mathbb{Q}_p$ be a quadratic irrational. For every integer $b$ with $|b|\leq d$, write
\[
\Delta-b^2=p^{e_b}u_b, \quad p\nmid u_b.
\]
Define
\[\mathcal{B}:=\left\{b\in\mathbb{Z}: |b|\leq d,\; e_b\geq2,\;
v_p(b+\delta)=0 \right\}\]
and
\begin{equation} \label{eq:bound-tau}
K:= \sum_{b \in \mathcal B} (e_b-1)\tau(u_b),
\end{equation}
where $\tau(u_b)$ denotes the number of positive divisors of $u_b$.
If there exists $n_0 > 1$ such that $N(\xi_n)<0$ for every $n\in[n_0,n_0+K],$ then the Browkin continued fraction expansion of $\alpha$ is periodic
with period of length at most $K$.

\end{Proposition}

\begin{proof}
Let $n\in[n_0,n_0+K]$. By definition, $N(\xi_n)=\frac{b_n^2-\Delta}{C_n^2},$ and $N(\xi_n)<0$ if and only if $b_n^2-\Delta<0.$
Thus, according to \eqref{eq:bncnrecurrence}, the $C_n$'s all have the same sign for $n\in[n_0,n_0+K]$ and $p^2\mid\Delta-b_n^2.$

Fix now a possible value of $b_n$, and write
\[
\Delta-b_n^2=p^{e_b}u_b, \quad e_b \geq 2, \quad p\nmid u_b.
\]
Since both $C_n$ and $C_{n-1}$ are divisible by $p$,  $|C_n|$ must be a divisor of $\Delta-b^2$ such
that
\[
p\mid |C_n| \quad\text{and}\quad p\mid\frac{\Delta-b^2}{|C_n|}.
\]
Since $e_b\geq2$, every such divisor can be uniquely written as
\[
|C_n|=p^jv, \quad 1\leq j\leq e_b-1, \quad
v\mid u_b.
\]
Hence, there are exactly
\[
(e_b-1)\tau(u_b)
\]
possible values for $|C_n|$. Since all the $C_n$'s have the same
sign on the interval under consideration, this is also an upper bound
for the number of possible values of $C_n$ for a fixed $b_n$.
We can further restrict the possible values of $b_n$. Indeed, for
$n\geq1$ we have
\[
-k_n
=
v_p(\alpha_n)
=
v_p(b_n+\delta)-k_n
\]
and therefore
\[
v_p(b_n+\delta)=0.
\]
Finally, since $N(\xi_n)<0$, we have $b_n^2<\Delta$. 
Therefore, since $p^2 \mid \Delta- b_n^2$, $v_p(b_n+\delta)=0$, $|b_n| \leq d$, we have  $b_n \in \mathcal B$.
It follows that the number of possible complete quotients $\alpha_n$
with negative norm is at most
\[
\sum_{b \in \mathcal B} (e_b-1)\tau(u_b).
\]
Hence, by the pigeonhole principle, there exist $\bar m,\bar n\in[n_0,n_0+K],$ $\bar m<\bar n,$ such that $\alpha_{\bar m}=\alpha_{\bar n}.$
Therefore, the Browkin continued fraction expansion of $\alpha$ is
periodic, and its period length satisfies $\lambda(\alpha)\leq K.$
\end{proof}

\begin{Remark}
The bound given in the previous proposition is the sharpest one
provided by the counting argument used in its proof, since it keeps
track of the exact number of divisors of $\Delta-b^2$ satisfying the
necessary divisibility conditions. However, it is not given in closed
form, since it involves the divisor function $\tau$ and the use of the set $\mathcal B$.
Although $\tau(m)$ can be computed exactly from the prime
factorization of $m$, no closed formula in terms of $m$ is
available. Thus, it is useful to derive explicit bounds
for $K$ by using upper bounds for the divisor function.
Considering the elementary estimate $\tau(m)\leq 2\sqrt{m}$ and also that $\mathcal B \subseteq \{ -d, \ldots, d\}$, we have
\begin{equation} \label{eq:bound-el}
K \leq\left\lfloor2
\sum_{b=-d}^d \sqrt{\Delta - b^2} \right\rfloor.
\end{equation}
A substantially sharper estimate can be obtained by using a result of
Nicolas and Robin~\cite{Nicolas}, who proved that
\[
\tau(m)\leq m^{\frac{C}{\log\log m}}, \quad C =1.0660186782\ldots ,
\]
for every $m\geq3$. 
\end{Remark}

With the same argument used in the previous proposition, one can also improve the bound in point 2 of Theorem \ref{thm:CMT}. In the next proposition, we generalize this result.
For the sake of simplicity, in the following we will always use the elementary estimates employed above, namely the bound
\[
\tau(m)\leq 2\sqrt{m}
\]
together with the inclusion
\[
\mathcal B\subseteq\{-d,\ldots,d\}.
\]
This leads to simpler explicit bounds and allows us to keep the statements and proofs more transparent. Naturally, sharper estimates can be obtained in each case by retaining the exact counting argument of Proposition~11, or by using stronger upper bounds for the divisor function.

\begin{Proposition}\label{prop:neg}
Let $\alpha\in\mathbb{Q}_p$ be a quadratic irrational. Let $I$ be a set of indices such that $N(\xi_i)<0$ for all $i \in I$. If $|I| $ is greater than  
\begin{equation*}
        K := \Big\lfloor 4\sum_{i=-d}^{d}\sqrt{\Delta-i^2}\Big\rfloor
\end{equation*}
then the Browkin continued fraction of $\alpha$ is periodic. Assume $I = \{i_1, \dots, i_K,\dots\}$, with $i_{j}<i_{j+1}$. Then, the period of the Browkin continued fraction of $\alpha$ has length at most 
\begin{equation*}
        \min_{j=1,\dots, |I|-K} i_{j+K}- i_j. 
\end{equation*}
\end{Proposition}
\begin{proof}
Let $n \in I$. By definition, it holds that $N(\xi_n)= \frac{b_n^2-\Delta}{p^{2k_n}c_n^2}$, hence $N(\xi_n)<0 \Leftrightarrow b_n^2 -\Delta <0$. According to \eqref{eq:bncnrecurrence}, for a fixed $b_n$,  $p^{k_n}c_n$ can take at most $4(\sqrt{\Delta-b_n^2})$ values. Since $b_n^2 < \Delta$ we have $|b_n| \le d$ for all $n \in I$, which means that $b_n$ can assume at most $2d+1$ values between $-d$ and $d$. Hence, the number of possible candidates for $\alpha_n$, for $n \in I$, is $K= \Big\lfloor 4\sum_{i=-d}^{d}\sqrt{\Delta-i^2}\Big\rfloor$. 
Let $\hat j$ be a positive integer such that $i_{\hat j + K} - i_{\hat j} = \min_{j=1,\dots, |I|-K} i_{j+K}- i_j$.
By observing that the number of $\alpha_n$ for which $N(\xi_n)<0$, for $n \in \{ i_{\hat j}, \ldots, i_{\hat j + K} \}$, is at least $K+1$, it follows that there exists $\Bar{m}, \Bar{n} \in \{ i_{\hat j}, \ldots, i_{\hat j + K} \}$ for which $\alpha_{\Bar{m}} = \alpha_{\Bar{n}}$, thus concluding the proof.
\end{proof}

Now, we deal with the case of positive norms. 

\begin{Proposition}\label{prop:norm_positive_1}
    Assume $\alpha \in \mathbb{Q}_p$ has a periodic Browkin continued fraction expansion and suppose that \begin{itemize}
        \item $N(\xi_n)>0$ for all $n$ in the period;
        \item there exists $\eta >0$ such that $|a_n|\ge 2 + \eta$, for all $n$ in the period.
    \end{itemize}
Then,
    \begin{equation}\label{eq:bound-eta}
D_n \le M, \quad M:= \left \lfloor \sqrt{\frac{\Delta}{\eta+\frac{\eta^2}{4}}}\right\rfloor
    \end{equation}
for all $n$ in the period and 
    \begin{equation*}
         \lambda(\alpha) \le 2M(2\left\lfloor \sqrt{M^2+\Delta}\right \rfloor+1).
    \end{equation*}
\end{Proposition}
\begin{proof}
    Since $\alpha$ is periodic, there is a finite set of possible values for $D_n$ in the period. Let $m$ be an index such that $D_m = \max_nD_n$ for $n$ in the period. We now prove that 
    \begin{equation*}
        |a_m| \le 2 \sqrt{1+\frac{\Delta}{D_m^2}}.
    \end{equation*}
    Indeed, since the norms are all positive, $C_n,C_{n-1}$ have opposite signs, therefore we can write \begin{equation}\label{eq:bn}
        b_n^2 = \Delta - C_nC_{n-1} = \Delta + D_nD_{n-1}.
    \end{equation}
    Moreover, according to \eqref{eq:bncnrecurrence} we have that $a_n = \frac{b_{n+1}+b_n}{C_n}$, and we define \begin{equation*}
        \Bar{a}_n := \frac{b_{n+1}+b_n}{D_n} = \pm a_n.
    \end{equation*}
    It follows that $b_{n+1}=\Bar{a}_nD_n-b_n$. Hence, we have

 \begin{align*}
        D_nD_{n+1} &= b_{n+1}^2-\Delta = (\Bar{a}_nD_n-b_n)^2-\Delta \\ &=  \Bar{a}_n^2D_n^2 - 2\Bar{a}_nb_nD_n + b_n^2-\Delta \\
        &= \Bar{a}_n^2D_n^2 - 2\Bar{a}_nb_nD_n+D_{n}D_{n-1}.
    \end{align*}
from which
\begin{equation}\label{eq:Dn1}
        D_{n+1} = \Bar{a}_n^2D_n-2\Bar{a}_nb_n+D_{n-1}.\end{equation}
Considering the index $m$ in \eqref{eq:bn}, the following relation is obtained 
    \begin{equation}\label{eq:bmDm}
    \left(\frac{b_m}{D_m}\right)^2 = \frac{D_{m-1}}{D_{m}} + \frac{\Delta}{D_m^2}\Rightarrow \left|\frac{b_m}{D_m}\right|  \le \sqrt{1+ \frac{\Delta}{D_m^2}}.
    \end{equation}
    Moreover, according to \eqref{eq:Dn1}, we get
    \begin{align}
        \frac{D_{m+1}}{D_{m}} = \frac{D_{m-1}}{D_{m}} -2\Bar{a}_m\frac{b_m}{D_m}+\Bar{a}_m^2.
    \end{align}
    By adding and subtracting the quantity $\frac{b_m^2}{D_m^2}$, it is possible to derive that \begin{equation*}
        \left(\Bar{a}_m- \frac{b_m}{D_m}\right)^2 = \frac{D_{m+1}}{D_{m}}+\left(\frac{b_m}{D_m}\right)^2- \frac{D_{m-1}}{D_{m}} =  \frac{D_{m+1}}{D_{m}}+\frac{\Delta}{D_m^2},
    \end{equation*} 
    which implies \begin{equation}\label{eq:ambmDm}
        \left|\Bar{a}_m-\frac{b_m}{D_m}\right| \le \sqrt{1+\frac{\Delta}{D_m^2}}
    \end{equation}
    since $D_m \geq D_{m+1}$.
    Now, from \eqref{eq:bmDm} and \eqref{eq:ambmDm}, it follows that \begin{equation*}
        2+ \eta \leq |a_m| =  |\Bar{a}_m|   \le 2 \sqrt{1+\frac{\Delta}{D_m^2}} \, ,
    \end{equation*}
    from which it is easy to verify 
    \begin{equation*}
        \eta+\frac{\eta^2}{4} \le \frac{\Delta}{D_m^2} \Rightarrow D_m \le \left\lfloor \sqrt{\frac{\Delta}{ \eta+\frac{\eta^2}{4}}}\right\rfloor =M.
    \end{equation*}
    By definition of $m$, this inequality also holds for all $n$ in the period. Finally, it is sufficient to observe that $b_n^2=D_{n-1}D_n+\Delta \le M^2 +\Delta$, which implies that $|b_n| \le \sqrt{M^2+\Delta}$, meaning that the number of possible pairs $(b_n,C_n)$ is bounded by \begin{equation*}
        2M\cdot(2\lfloor \sqrt{M^2+\Delta} \rfloor+1).
    \end{equation*}
\end{proof}

\begin{Corollary} \label{cor:np-eta}
    Let $\alpha \in \mathbb{Q}_p$ quadratic irrational and let $n_0, K \in \mathbb{N}$, $\eta >0$ such that \begin{itemize}
        \item $N(\xi_n)>0$ for all $n \in [n_0, n_0+K]$;
        \item $\displaystyle \max_{n_0 \le n \le n_0+K} D_n = \max_{n_0 < n < n_0+K} D_n$;
        \item $|a_n| \ge 2+ \eta$ for all $n \in [n_0, n_0+K]$;
    \end{itemize}
    then, $D_n \le M := \left \lfloor \sqrt{\frac{\Delta}{\eta+\frac{\eta^2}{4}}}\right\rfloor$ for all $ n \in [n_0, n_0+K]$. Moreover, if $$K > 2M\cdot(2\left\lfloor \sqrt{M^2+\Delta} \right \rfloor+1)$$ then $\alpha$ is periodic with $$\lambda(\alpha) \le 2M\cdot(2\lfloor  \sqrt{M^2+\Delta} \rfloor+1).$$
\end{Corollary}
\begin{proof}
    The proof follows similarly to the proof of Proposition \ref{prop:norm_positive_1}, observing that $|a_j| \le 2 \sqrt{1+\frac{\Delta}{D_j^2}}$ for all $j$ such that $D_j \ge D_{j-1}$ and $D_j \ge D_{j+1}$.
\end{proof}

In what follows, we relax the condition on the partial quotients, requiring only that $|a_n| > 2$. The resulting statements clearly include the case discussed above. We nevertheless state the two cases separately, since the bounds obtained in Proposition \ref{prop:norm_positive_1} and Corollary \ref{cor:np-eta} are sharper than those derived below.

\begin{Lemma}\label{lemma:Dn}
Fix $n \in \mathbb{N}$, and let $N(\xi_n),N(\xi_{n+1})>0$, $D_n \ge D_{n-1}$, $D_n \ge D_{n+1}$, $|a_n|>2$. Then, $D_n < \Delta$.
\end{Lemma}
\begin{proof}
    From the previous proofs we know that $|a_n| \le 2 \sqrt{1+\frac{\Delta}{D_n^2}}$. Hence, \begin{equation*}
        0< |a_n| -2 \le 2\left(\sqrt{1+\frac{\Delta}{D_n^2}} -1\right) < \frac{\Delta}{D_n^2}.
    \end{equation*} 
Moreover, from 
\begin{align*}
    |a_n|-2 = \frac{|\Tilde{a}_n|-2p^{k_n}}{p^{k_n}} \ge \frac{1}{p^{k_n}}, \quad 
    D_n \ge p^{k_n} 
\end{align*}
we get
\begin{equation*}
        \frac{1}{D_n} \le \frac{1}{p^{k_n}} \le |a_n|-2 < \frac{\Delta}{D_n^2}
\end{equation*}
which concludes the proof.
\end{proof}

\begin{Proposition}\label{prop:Npos2}
Assume $\alpha \in \mathbb{Q}_p$ has a periodic Browkin continued fraction expansion and suppose  
\begin{itemize}
        \item $N(\xi_n)>0$ for all $n$ in the period;
        \item $|a_n| > 2$, for all $n$ in the period;
\end{itemize}
    then, for all $n$ in the period, 
    \begin{equation} \label{eq:bound-a2}
    D_n < \Delta
    \end{equation}
    and 
    $$\lambda(\alpha)\le 2 (\Delta-1)(2\lfloor \sqrt{ ( \Delta-1)^2+\Delta }\rfloor+1).$$
\end{Proposition}
\begin{proof}
    The proof follows as the proof of Proposition \ref{prop:norm_positive_1} applying Lemma \ref{lemma:Dn} to bound $D_n$'s in the period.
\end{proof}
Similarly to Corollary \ref{cor:np-eta}, we can also obtain the following corollary from the previous proposition. 
\begin{Corollary}
    Let $\alpha \in \mathbb{Q}_p$ be a quadratic irrational and let $n_0, K \in \mathbb{N}$, such that 
    \begin{itemize}
        \item $N(\xi_n)>0$ for all $n \in \{n_0,n_0+1,\dots,n_0+K\}$;
        \item $\displaystyle \max_{n_0 \le n \le n_0+K} D_n = \max_{n_0 < n < n_0+K} D_n$;
        \item $|a_n| > 2$ for all $n \in \{n_0,n_0+1,\dots,n_0+K\}$;
    \end{itemize}
    then, $D_n < \Delta$ for all $n \in \{n_0,n_0+1,\dots,n_0+K\}$. Moreover, if $$K > 2 (\Delta-1)(2\lfloor \sqrt{ ( \Delta-1)^2+\Delta }\rfloor+1)$$ then $\alpha$ is periodic with $$\lambda(\alpha)\le 2 (\Delta-1)(2\lfloor \sqrt{ ( \Delta-1)^2+\Delta }\rfloor+1).$$
\end{Corollary}

In the following, we will treat the case where $1 < |a_n| < 2$.

\begin{Lemma}\label{lemma:b}
Assume that $\alpha\in\mathbb Q_p$ has a periodic Browkin
continued fraction expansion and suppose that
\[
N(\xi_n)>0,
\qquad
1<|a_n|<2
\]
for all $n$ in the period. Let $m$ and $\ell$ be indices such that
$D_m=\max_n D_n$ and $D_\ell=\min_n D_n.$ Then
\[
b_mb_{m+1}>0, \quad b_\ell b_{\ell+1}<0, \quad D_m > \frac{D_\ell+\sqrt{D_\ell^2+4\Delta}}{2}.
\]
Moreover,
\[
\max\{D_{\ell-1},D_{\ell+1}\} > (1+|a_\ell|)^2D_\ell > 4D_\ell, \quad \frac{\max_nD_n}{\min_nD_n}>4.
\]
\end{Lemma}
\begin{proof}
Since $N(\xi_n)>0$ for every $n$ in the period, by
\eqref{eq:bncnrecurrence} we have
\[
b_n^2=\Delta+D_nD_{n-1}, \quad |b_n+b_{n+1}|=|a_n|D_n.
\]
We prove that $b_mb_{m+1}>0$. Assume, by contradiction, that
$b_mb_{m+1}<0$. Then
\[
|b_m+b_{m+1}|
=
\bigl||b_m|-|b_{m+1}|\bigr|
\]
and we obtain
\[
|a_m|D_m =\bigl||b_m|-|b_{m+1}|\bigr|= \frac{|b_m^2-b_{m+1}^2|}
{|b_m|+|b_{m+1}|}=D_m\frac{|D_{m-1}-D_{m+1}|}{|b_m|+|b_{m+1}|}.
\]
Since $D_m\geq D_{m-1},D_{m+1}$, we have
\[
|b_m|
=
\sqrt{\Delta+D_mD_{m-1}}
>
D_{m-1},
\]
and similarly $|b_{m+1}|>D_{m+1}.$
Therefore,
\[
|b_m|+|b_{m+1}|
>
D_{m-1}+D_{m+1}
\geq
|D_{m-1}-D_{m+1}|,
\]
and consequently $|a_m|<1,$ which contradicts the assumption $|a_m|>1$. Hence, $b_mb_{m+1}>0.$

We now prove that $b_\ell b_{\ell+1}<0$. Indeed, if
$b_\ell b_{\ell+1}>0$, then $|b_\ell+b_{\ell+1}|=
|b_\ell|+|b_{\ell+1}|.$
Since $D_{\ell-1},D_{\ell+1}\geq D_\ell,$ we obtain
\[
|b_\ell|=\sqrt{\Delta+D_\ell D_{\ell-1}}>D_\ell, \quad |b_{\ell+1}|
=\sqrt{\Delta+D_\ell D_{\ell+1}}>D_\ell.
\]
Thus,
\[
|a_\ell|D_\ell
=
|b_\ell|+|b_{\ell+1}|
>
2D_\ell,
\]
which gives $|a_\ell|>2$, a contradiction. 

We next derive a lower bound for $D_m$. Since
$b_mb_{m+1}>0$, we have
\[
|a_m|D_m=\sqrt{\Delta+D_mD_{m-1}}+\sqrt{\Delta+D_mD_{m+1}}
\geq2\sqrt{\Delta+D_mD_\ell}.
\]
From $|a_m|<2$, it follows that
\[
2D_m
>
2\sqrt{\Delta+D_mD_\ell},
\]
and therefore
\[
D_m^2-D_\ell D_m-\Delta>0.
\]
Thus,
\[
D_m>
\frac{D_\ell+\sqrt{D_\ell^2+4\Delta}}{2}.
\]

Finally, we prove the last inequalities. Assume without loss of generality, that $D_{\ell+1}\geq D_{\ell-1}$.
Since $b_\ell b_{\ell+1}<0$, 
\[
|a_\ell|=\frac{D_{\ell+1}-D_{\ell-1}}{\sqrt{\Delta+D_{\ell}D_{\ell+1}}+\sqrt{\Delta+D_{\ell}D_{\ell-1}}}<
\frac{\sqrt{D_{\ell+1}}-\sqrt{D_{\ell-1}}}{\sqrt{D_\ell}}.
\]
It follows that
\[
D_{\ell+1}>(1+|a_\ell|)^2D_\ell.
\]
We obtain
\[
\max\{D_{\ell-1},D_{\ell+1}\}
>
(1+|a_\ell|)^2D_\ell.
\]
Since $|a_\ell|>1$, this yields
\[
\max\{D_{\ell-1},D_{\ell+1}\}>4D_\ell.
\]
Finally,
\[
\max_nD_n
\geq
\max\{D_{\ell-1},D_{\ell+1}\}
>
4D_\ell
=
4\min_nD_n,
\]
and therefore
\[
\frac{\max_nD_n}{\min_nD_n}>4.
\]
\end{proof}

\begin{Proposition} \label{prop:Npos-a1}
Assume that $\alpha\in\mathbb Q_p$ has a periodic Browkin
continued fraction expansion and suppose that $N(\xi_n)>0$
for every $n$ in the period. Let $m$ be an index such that $D_m=\max_n D_n,$ and assume that $|a_m|>1$. Define
\[
L:=\frac{|a_m|^2}{4}-\frac{D_{m-1}+D_{m+1}}{2D_m}.
\]
If $L>0$, then
\[
D_n\leq M, \quad M:=\left\lfloor
\sqrt{\frac{\Delta}{L}}
\right\rfloor,
\]
for every $n$ in the period. Consequently,
\[
\lambda(\alpha)
\leq
2M
\left(
2\left\lfloor\sqrt{M^2+\Delta}\right\rfloor+1
\right).
\]
\end{Proposition}

\begin{proof}
Since \(D_m\) is maximal and \(|a_m|>1\), the argument used
in the proof of Lemma \ref{lemma:b} to show that \(b_m b_{m+1}>0\) applies also in the present setting, as it only relies on the local assumptions
\[
D_m \geq D_{m-1}, D_{m+1}
\qquad\text{and}\quad
|a_m|>1.
\]
Hence, $b_m b_{m+1}>0.$
It follows that
\[
|a_m|D_m
=
|b_m|+|b_{m+1}|
\]
and
\[
|a_m|^2D_m^2
=
\left(|b_m|+|b_{m+1}|\right)^2
\leq
2\left(b_m^2+b_{m+1}^2\right).
\]
Then, we obtain
\[
|a_m|^2D_m^2
\leq
4\Delta
+
2D_m(D_{m-1}+D_{m+1})
\]
from which
\[
\left(\frac{|a_m|^2}{4}-\frac{D_{m-1}+D_{m+1}}{2D_m}\right)D_m^2\leq\Delta.
\]
By definition of $L$ and assuming $L > 0$ we get
\[
D_m\leq\left\lfloor\sqrt{\frac{\Delta}{L}}\right\rfloor
=M.
\]
Finally, we have
\[
|b_n^2-\Delta|=D_nD_{n-1}\leq M^2,
\]
and consequently
\[
|b_n|\leq\left\lfloor\sqrt{M^2+\Delta}\right\rfloor.
\]
Thus, there are at most
\[
2M\left(2\left\lfloor\sqrt{M^2+\Delta}\right\rfloor+1
\right)
\]
possible pairs $(b_n,C_n)$ in the period. 
\end{proof}

\begin{Corollary}
Let $\alpha\in\mathbb{Q}_p$ be a quadratic irrational and let
$n_0,K\in\mathbb{N}$ be such that $N(\xi_n)>0$ for all    $n\in[n_0, n_0 + K]$ and
\[\max_{n_0\leq n\leq n_0+K} D_n=\max_{n_0<n<n_0+K} D_n.\]
Let $m\in[n_0, n_0 + K]$ be such that $D_m=\max_{n_0\leq n\leq n_0+K}D_n$ and assume that $|a_m|>1$. Define
\[ L:=\frac{|a_m|^2}{4}-\frac{D_{m-1}+D_{m+1}}{2D_m}.
\]
If $L>0$, then
\[D_n\leq M, \quad M:=\left\lfloor\sqrt{\frac{\Delta}{L}}\right\rfloor,\]
for all $n\in[n_0, n_0+K]$. Moreover, if
\[K>2M\left(
2\left\lfloor\sqrt{M^2+\Delta}\right\rfloor+1
\right),
\]
then $\alpha$ is periodic with
\[
\lambda(\alpha)\leq2M\left(2\left\lfloor\sqrt{M^2+\Delta}\right\rfloor+1\right).
\]
\end{Corollary}

\begin{proof}
The proof follows similarly to the proof of Proposition \ref{prop:Npos-a1}, observing that the argument is local and only requires
\[
D_m\geq D_{m-1},D_{m+1},
\qquad
|a_m|>1.
\]
Hence, if $L>0$, we obtain
\[
D_m\leq
\left\lfloor
\sqrt{\frac{\Delta}{L}}
\right\rfloor
=M.
\]
By the choice of $m$, the same bound holds for all
$n\in\{n_0,\ldots,n_0+K\}$.
Moreover,
\[
b_n^2
=
\Delta+D_nD_{n-1}
\leq
\Delta+M^2
\]
for every $n\in\{n_0+1,\ldots,n_0+K\}.$ Thus, there are at
most
\[
2M
\left(
2\left\lfloor\sqrt{M^2+\Delta}\right\rfloor+1
\right)
\]
possible pairs $(b_n,C_n)$. The conclusion follows by the
pigeonhole principle.
\end{proof}

The previous results provide sufficient conditions for periodicity by deriving upper bounds for the quantities $D_n$ from suitable assumptions on the partial quotients. We now follow a different approach, showing that, when the complete quotients have eventually positive norm, controlling the variation of consecutive values $D_n$ is itself sufficient to force periodicity.

\begin{Proposition}
Let $\alpha\in\mathbb Q_p$ be a quadratic irrational and let
$H\geq0$. Define
\[K:=4\sum_{\substack{h\in\mathbb Z\\
                 |h|\leq H\\
                 p\mid h}}
\tau\left(\left|4\Delta-h^2\right|\right).
\]
Assume that there exists $n_0>1$ such that, for every $n\in[n_0,n_0+K],$ we have $N(\xi_n)>0$ and $|D_n-D_{n-1}|\leq H.$ Then the Browkin continued fraction expansion of $\alpha$ is
periodic, with period of length at most $K$.
\end{Proposition}
\begin{proof}
For every $n\in[n_0,n_0+K]$, set
\[
h_n:=D_n-D_{n-1},
\qquad
s_n:=D_n+D_{n-1}.
\]
Since $N(\xi_n)>0$, we have $b_n^2=\Delta+D_nD_{n-1}.$
Moreover,
\[
4D_nD_{n-1}=s_n^2-h_n^2,
\]
and therefore
\[
4b_n^2=4\Delta+s_n^2-h_n^2.
\]
It follows that
\[
(2b_n-s_n)(2b_n+s_n)=4\Delta-h_n^2.
\]
Fix a possible value $h$ of $h_n$. Put
\[
x_n:=2b_n-s_n,
\qquad
y_n:=2b_n+s_n.
\]
Then
\[
x_ny_n=4\Delta-h^2.
\]
There are at most
\[
2\tau\left(\left|4\Delta-h^2\right|\right)
\]
possible integer values for $x_n$, and $y_n$ is then uniquely
determined. Moreover,
\[
b_n=\frac{x_n+y_n}{4},
\qquad
s_n=\frac{y_n-x_n}{2},
\qquad
D_n=\frac{s_n+h}{2}.
\]
Thus, for a fixed $h$, there are at most
\[
2\tau\left(\left|4\Delta-h^2\right|\right)
\]
possible pairs $(b_n,D_n)$.
For $n>1$, both $D_n$ and $D_{n-1}$ are divisible by $p$.
Consequently, $p\mid h_n.$
Finally, for each pair $(b_n,D_n)$ there are at most two possible
values of $C_n$, namely $C_n=\pm D_n$. Hence, the total number of
possible complete quotients in the interval is at most
\[
4
\sum_{\substack{h\in\mathbb Z\\
                 |h|\leq H\\
                 p\mid h}}
\tau\left(\left|4\Delta-h^2\right|\right)=K.
\]

Since the interval contains $K+1$ complete quotients, two of
them must coincide. Therefore, the Browkin continued fraction
expansion of $\alpha$ is periodic, and its period length is at
most $K$.
\end{proof}

\begin{Corollary} \label{cor:diff}
Let $\alpha\in\mathbb Q_p$ be a quadratic irrational. Assume that
there exists $n_0>1$ such that $N(\xi_n)>0$ for every $n\geq n_0$. If the sequence $\left(|D_n-D_{n-1}|\right)_{n\geq n_0}$ is bounded, then the Browkin continued fraction expansion of
$\alpha$ is periodic.
\end{Corollary}

The preceding results can be used to derive practical sufficient conditions for non-periodicity, thereby providing useful tools in the search for counterexamples to the analogue of Lagrange's theorem for Browkin continued fractions. Indeed, as we will discuss in Section \ref{sec:sq}, it seems unlikely that every quadratic irrational has a periodic expansion through the Browkin algorithm. Proposition~\ref{prop:neg} shows that any possible counterexample can have only finitely many complete quotients of negative norm and must therefore have eventually positive norms. In the following, we summarize the previous results in terms of conditions for non--periodicity. 

\begin{Proposition}
Let $\alpha = \sqrt{\Delta}\in\mathbb{Q}_p$, where $\Delta>0$ is a
non-square integer, and assume that there exists $n_0 \geq 2$ such that \(N(\xi_n)>0\) for every \(n\geq n_0.\)
\begin{enumerate}
\item Assume that there exists $\eta>0$ such that
\(
|a_n|\geq 2+\eta\) for every \(n\geq n_0.\)
Define
\[
M
:=
\left\lfloor
\sqrt{
\frac{\Delta}
{\eta+\eta^2/4}
}
\right\rfloor \quad
\text{and} \quad
K:=2M\left(2\left\lfloor\sqrt{M^2+\Delta}\right\rfloor+1\right).
\]
If either
\[
D_n>M
\quad\text{for some }n\geq n_0,
\]
or
\[
\alpha_{n+2}\neq\alpha_{2}
\qquad
\text{for every }1\leq n\leq K,
\]
then the Browkin continued fraction expansion of
$\sqrt{\Delta}$ is not periodic.

\item Assume that \(|a_n|>2\) for every $n\geq n_0$.
Define
\[
K
:=
2(\Delta-1)
\left(
2\left\lfloor
\sqrt{(\Delta-1)^2+\Delta}
\right\rfloor+1
\right).
\]
If either
\[
D_n>\Delta
\qquad\text{for some }n\geq n_0,
\]
or
\[
\alpha_{n+2}\neq\alpha_{2}
\qquad
\text{for every }1\leq n\leq K,
\]
then the Browkin continued fraction expansion of
$\sqrt{\Delta}$ is not periodic.

\item Assume that $1<|a_n|<2$ for every $n\geq n_0$ and
\[
b_nb_{n+1}\geq0
\qquad
\text{for every } n\geq n_0.
\]
Then the Browkin continued fraction expansion of
$\sqrt{\Delta}$ is not periodic.

\item Assume that $1<|a_n|<2$ for every $n\geq n_0$ and
\[
\max\{D_{n-1},D_{n+1}\}
\leq
(1+|a_n|)^2D_n
\qquad
\text{for every } n>n_0.
\]
Then the Browkin continued fraction expansion of
$\sqrt{\Delta}$ is not periodic.
\item Assume that there exists $\eta>0$ such that, for every
$n\geq n_0$ satisfying $D_n\geq D_{n-1},D_{n+1},$ we have $|a_n|>1$ and
\[
\frac{|a_n|^2}{4}-\frac{D_{n-1}+D_{n+1}}{2D_n}\geq \eta.
\]
Define
\[
M:=\left\lfloor\sqrt{\frac{\Delta}{\eta}}\right\rfloor, \quad K:=
2M\left(2\left\lfloor\sqrt{M^2+\Delta}\right\rfloor+1\right).
\]
If either
\[
D_n>M
\]
for some $n\geq n_0$, or
\[
\alpha_{n+2}\neq\alpha_2
\]
for every $1\leq n\leq K$, then the Browkin continued fraction
expansion of $\sqrt{\Delta}$ is not periodic.
\end{enumerate}
\end{Proposition}
\begin{proof}
It follows from Propositions \ref{prop:norm_positive_1}, \ref{prop:Npos2}, \ref{prop:Npos-a1} and Lemma \ref{lemma:b} recalling that, for the square root of an integer, the preperiod (in case of periodic expansions) has length 2.
\end{proof}

Finally, we recall the following result from \cite{CMT} from which we derive another sufficient condition for non--periodicity.
\begin{Lemma}[Proposition 4.1 and Corollary 4.6]\label{lemma:boundedbn}
Let $\alpha \in \mathbb Q_p$ be a quadratic irrational. The expansion of the Browkin continued fraction of $\alpha$ is periodic if and only if $\{|b_n|\}_{n \ge 0}$ is bounded from the above.
\end{Lemma}
Before stating the next proposition we recall that, for $n \ge 1$, it holds $k_n \ge 1$, $C_n \in \mathbb{Z}$, $D_n \ge p$. Moreover, $$|a_nD_n|= |\Tilde{a}_nc_n| \ge 1.$$ 
We define with $\varepsilon_n$ the sign such that $D_n=\varepsilon_nC_n$
\begin{Proposition}
    Let $\alpha \in \mathbb{Q}_p$ be a quadratic irrational such that there exists $n_0\ge1$ for which 
    \begin{enumerate}
        \item[(i)] $N(\xi_n)>0$ for every $n\ge n_0$;
        \item[(ii)] the partial quotients $a_n$ have constant sign for every $n\ge n_0$.
    \end{enumerate}
    Then the Browkin continued fractions expansion of $\alpha$ is not periodic.
\end{Proposition}
\begin{proof}
   Define $g_n = (-1)^nb_n$. Then, 
   \begin{equation*}
       g_{n+1}-g_n=(-1)^{n+1}(b_{n+1}+b_n)=(-1)^{n+1}\Tilde{a}_nc_n.
   \end{equation*}
   Hence, $|g_{n+1}-g_n| \ge 1$. Moreover, since $N(\xi_n)>0$ for every $n\ge n_0$ the values $c_{n-1},c_n$ have opposite signs, which means that for $n \ge n_0$ it holds that $\varepsilon_n=\varepsilon_{n_0}(-1)^{n-n_0}$. Therefore, \begin{equation*}
       g_{n+1}-g_n=(-1)^{n+1}\varepsilon_{n_0}(-1)^{n-n_0}\Tilde{a}_n|c_n|=(-1)^{1-n_0}\varepsilon_{n_0}\Tilde{a}_n|c_n|.
   \end{equation*} 
   which means that all the increments $g_{n+1}-g_n$ have the same sign, since $(-1)^{1-n_0}\varepsilon_{n_0}$ is fixed and the partial quotients have the same sign according to $(ii)$. 
   Hence, all the increments \(g_{n+1}-g_n\) have the same sign
and satisfy $|g_{n+1}-g_n|\geq 1.$
Therefore, \(g_n\to+\infty\) or \(g_n\to-\infty\), according
to the sign of the increments. In either case, $|g_n|\to\infty.$ Since \(|g_n|=|b_n|\), it follows that the sequence
\((|b_n|)_{n\geq 0}\) is unbounded. By Lemma \ref{lemma:boundedbn}, the Browkin continued fraction expansion of \(\alpha\) is therefore not periodic.

\end{proof}

\section{Some computational considerations} \label{sec:comp}

In this section, we discuss some computational aspects and numerical examples related to the previous results.

\subsection{Nice continued fractions} \label{sec:comp-nice}\ \\

Theorems \ref{thm:C} and \ref{thm:B} provide an explicit method for extending a continued fraction $[a_0,\ldots,a_{t-2}]$ by appending a partial quotient $a_{t-1}$ in order to obtain a nice continued fraction. In Algorithm \ref{alg}, we present this method in pseudocode.

\begin{algorithm}[H]
\caption{Completion of $[a_0, \ldots, a_{t-2}]$ to obtain a nice continued fraction}
\label{alg}
\begin{algorithmic}[1]
\Require prime $p$; Browkin continued fraction $[a_0, \ldots, a_{t-2}]$, with $|a_0| < p/4$ and $|a_0|_p > 1$
\Ensure partial quotient $a_{t-1}$ such that $[a_0, \ldots, a_{t-2}, a_{t-1}]$ is nice
\State $d \gets \tilde A_{t-3} - \tilde B_{t-3} \tilde B_{t-2}^{-1} \tilde A_{t-2} \pmod{\tilde A_{t-2}^2}$
\State $\varepsilon \gets \text{sign}(\tilde A_{t-3})$
\State $r \gets 1$, $n \gets 1$
\While{$r$ is not prime}
\State $r \gets \varepsilon pd + n \tilde A_{t-2}^2$
\State $n \gets n + 1$
\EndWhile
\State $\ell \gets \lambda(\tilde A_{t-2}^2)$, $h \gets \text{ord}_{\tilde A_{t-2}^2}(p)$
\State $\eta \gets \max \left( \frac{4}{p} |\tilde A_{t-2}|, p^{k_{t-2}} |\tilde A_{t-3}| - \frac{p}{2} |\tilde A_{t-2}| \right)$, $\mu \gets p^{k_{t-2}} |\tilde A_{t-3}| + \frac{p}{2} |\tilde A_{t-2}|$
\State $\eta' \gets \log_p(\eta) - \log_p(r) + 1 - k_{t-2}$, $\mu' \gets \log_p(\mu) - \log_p(r) + 1 - k_{t-2}$
\For{$n_1=0,1,2,\ldots$}
    \State $n_2 \gets \left\lfloor\frac{n_1 \ell \log_p(r) - \mu'}{h}\right\rfloor+1$
    \If{$\eta'<n_1 \ell \log_p(r) -n_2 h < \mu'$}
        \State \Return $(n_1,n_2)$ and \textbf{end for}
    \EndIf
\EndFor
\State $e \gets 1 + n_1 \ell$, $k_{t-1} \gets 1 - k_{t-2} + n_2 h$
\State $\tilde a_{t-1} \gets \cfrac{\varepsilon r^e - p^{k_{t-1}+k_{t-2}}\tilde A_{t-3}}{\tilde A_{t-2}}$
\State $a_{t-1} \gets \cfrac{\tilde a_{t-1}}{p^{k_{t-1}}}$
\State \Return $a_{t-1} $
\end{algorithmic}
\end{algorithm}
Although the method presented in Algorithm \ref{alg} is explicit, it is far from being efficient. 
By the prime number theorem, the average number of iterations required to find a prime $r$ of size $X$ in the arithmetic progression considered in Line 5   is $\frac{\varphi(|\tilde A_{t-2}|)}
{|\tilde A_{t-2}|}\log X.$
Thus, taking $X\approx |pd+\tilde A_{t-2}^{,2}|$, the expected number of trials is approximately
\[
\frac{\varphi(|\tilde A_{t-2}|)}
{|\tilde A_{t-2}|}
\log\left(|pd+\tilde A_{t-2}^{2}|\right).
\]
As for the loop in line 11, the number of values of $n_1$ that need to be tested should be approximately $\frac{h}{\mu' - \eta'}$. However, the main problem, from a computational point of view, occurs in Line 8, where the Carmichael function and the multiplicative order are computed. Indeed, their exact computation is closely related to the integer factorization problem. Consequently, no polynomial-time algorithms are known for this step and the best methods have subexponential complexity in the bit length of the input. Moreover, the values of $\tilde{a}_{t-1}$ and $k_{t-1}$ produced by the algorithm can be extremely large. For instance, when $p=7$, completing the continued fraction $\left[\frac{1}{7}, \frac{1}{7}\right]$ by means of the algorithm yields a partial quotient for which $k_{t-1} = 1 - k_{t-2} + n_2 h = 1 - 1 + 3646 \cdot 100$. The resulting numerators and denominators are of order $10^{308122}$. Nevertheless, much smaller values can be used to obtain a nice continued fraction. For example,
$\left[ \frac{1}{7}, \frac{1}{7}, \frac{11}{7} \right]$ or $\left[ \frac{1}{7}, \frac{1}{7}, \frac{12}{7} \right]$.

In Table \ref{tab:nice}, we present some examples of nice and non--nice continued fractions of the form $\left[ \frac{1}{p}, \ldots, \frac{1}{p}, \frac{a}{p} \right]$. Our experiments suggest that a continued fraction $\left[ \frac{1}{p}, \ldots, \frac{1}{p} \right]$ can always be completed to a nice continued fraction by appending a partial quotient of $p$--adic valuation $-1$. Such partial quotients, however, cannot be obtained through our algorithm. Therefore, the problem of finding a simpler and more efficient construction of nice continued fractions remains open.

In Table \ref{tab:nice}, for $t = 2, 3, 4, 5$, we report the values of $a$ for which $\left[ \frac{1}{7}, \ldots, \frac{1}{7}, \frac{a}{7} \right]$ is nice, as well as those for which it is not nice. Here, $t$ denotes the length of the initial continued fraction $\left[ \frac{1}{7}, \ldots, \frac{1}{7} \right]$. We observed similar results also for different primes. We performed some experiments also for larger values of $t$ even if in this case we were not able to check all the possible values of $a$ and we restricted our analysis to some random values. This is because, as $t$ increases, checking condition $c$ in Definition \ref{def:niceCF} becomes very time-consuming, since it involves the computation of discrete logarithms.

\begin{table}[H]
\centering
\caption{Values of $a$ for which $\left[ \frac{1}{p}, \ldots, \frac{1}{p}, \frac{a}{p} \right]$ is nice or non-nice continued fraction for various lengths $t$ of $\left[ \frac{1}{p}, \ldots, \frac{1}{p} \right]$}
\label{tab:nice}

\renewcommand{\arraystretch}{1.35}

\begin{tabularx}{\textwidth}{
    @{}
    >{\bfseries}l
    >{\raggedright\arraybackslash}X
    @{}
}
\toprule
$t=2$
&
Values of $a$
\\

nice & \(\{-18,-17,-16,-15,-13,-11,-9,-8,-6,-5,4,5,6,8,9,11,\)
\newline
\(\phantom{\{} 12,13,15,16,17 \}\)
\\
non-nice & \(\{-12,-10,-4,-3,-2,-1,1,2,3,10,18\}\)
\\
\bottomrule

\toprule
$t=3$
&
Values of $a$
\\

nice & \(\{ -18, -17, -16, -13, -12, -11, -9, -8, -6, -1,
 1, 3, 4, 8, 9, 11,\)
 \newline \(\phantom{\{} 12, 13, 15, 16, 17, 18 \}\)
\\
non-nice & \(\{-15, -10, -5, -4, -3, -2, 2, 5, 6, 10\}\)
\\
\bottomrule

\toprule
$t=4$
&
Values of $a$
\\

nice & \(\{-17, -16, -10, -8, -6, 8, 10, 12, 13, 16, 18, 3\}\)
\\
non-nice & \(\{-18, -15, -13, -12, -11, -9, -5, -4, -3, -2, -1, 1, 2, 4, 5, 6, \) \newline
\(\phantom{\{} 9, 11, 15, 17\}\)
\\
\bottomrule

\toprule
$t=5$
&
Values of $a$
\\

nice & \(\{-12, -11, -9, -6, -4, -3, -2, 1, 3, 6, 9, 10, 11, 12, 15, 16, 17, 18\}\)
\\
non-nice & \(\{-18, -17, -16, -15, -13, -10, -8, -5, -1, 2, 4, 5, 8, 13\}\)
\\
\bottomrule

\end{tabularx}
\end{table}

\subsection{Expansions for quadratic irrationals} \label{sec:sq} \ \\

We begin this section by examining the bounds given in Theorem \ref{thm:CMT} and the improvement provided by Proposition \ref{prop:improv} for quadratic irrationals with periodic expansions whose complete quotients all have negative norm throughout the period. In Table \ref{tab:neg}, we collect the results obtained in $\mathbb Q_7$ for several quadratic irrationals with different period lengths.
We observe that the bound in \eqref{eq:bound-CMT} is very large compared with the actual period length. The bound in \eqref{eq:bound-el}, obtained by using the elementary estimates for both the divisor function $\tau$ and the set $\mathcal B$, is also far from the numerical values, even if significantly better than \eqref{eq:bound-CMT}. On the other hand, the bound provided by \eqref{eq:bound-tau} is more accurate and appears to reflect the actual period length much more closely.

\begin{table}[H]
\caption{Quadratic irrationals with periodic expansion and negative norms in $\mathbb Q_7$.}
\label{tab:neg}
\centering
\begin{tabular}{ccccc}
\hline
$\alpha$ & Bound \eqref{eq:bound-CMT} & Bound \eqref{eq:bound-el} & Bound \eqref{eq:bound-tau} & $\lambda(\alpha)$ \\
\hline
$\frac{1+\sqrt{50}}{7}$ & 471 & 157 & 1 & 1\\
$\frac{-7+\sqrt{53}}{4}$& 516 & 167 & 1 & 1\\
$\frac{1+\sqrt{99}}{14}$& 1312 & 307 & 2 & 2 \\
$\frac{-25+\sqrt{674}}{7}$& 23325 & 2112 & 3 & 3 \\
$\frac{-25+\sqrt{870}}{7}$& 34221 & 2735 & 6& 4\\
$\frac{-26+\sqrt{872}}{7}$& 34339 & 2741 & 5& 4\\
$\frac{-46+\sqrt{2802}}{7}$& 197751 & 8796 & 10 & 6\\
$\frac{-54+\sqrt{3014}}{7}$& 220617 & 9462& 8 & 6\\
$\frac{-173+\sqrt{29978}}{7}$&6920569 &94183 & 33 & 7\\
\hline
\end{tabular}

\end{table}

Theorem \ref{thm:CMT} and Propositions \ref{prop:improv}, \ref{prop:neg} show that complete quotients with negative norm play a crucial role in forcing periodicity. In particular, Proposition \ref{prop:neg} implies that any possible counterexample to an analogue of Lagrange's theorem for Browkin continued fractions can contain only finitely many complete quotients with negative norm. Therefore, if such a counterexample exists, its complete quotients must eventually have positive norm. This naturally shifts our attention to the positive-norm case, which represents the only possible situation where non-periodic quadratic irrationals may occur. 

Let us consider now $p=5$. In Table \ref{tab:a2}, we provide some examples of quadratic irrationals satisfying the assumptions of Propositions \ref{prop:norm_positive_1} and \ref{prop:Npos2}. We observe that the bound on the values $D_n$ within the period provided by Proposition \ref{prop:norm_positive_1} is very accurate. By contrast, the corresponding bound on the period length is much less sharp, although it can be significantly improved by exploiting the same counting argument used in Proposition \ref{prop:improv}.

\begin{table}[H]
\caption{Quadratic irrationals with periodic expansion, positive norms and $|a_n| > 2$ in $\mathbb Q_5$.}
\label{tab:a2}
\centering
\begin{tabular}{cccccccc}
\hline
$\alpha$ & $\lambda(\alpha)$ & $\eta$ & $\max D_n$ & Bound \eqref{eq:bound-eta} & Bound \eqref{eq:bound-a2} & Best bound for $\lambda(\alpha)$ \\
\hline
$\frac{6+\sqrt{11}}{5}$ & 2            & 2/5 & 5 & 5   & 11 & 2\\
$\frac{11+\sqrt{21}}{10}$& 2           & 1/5 & 10 & 10 & 21 & 2 \\
$\frac{26+\sqrt{51}}{25}$& 2           & 2/25 & 25 & 25 & 51 & 2 \\
$\frac{576+\sqrt{74976}}{535}$& 6      & 1/5 & 535 & 597 & 74976 & 78 \\
$\frac{3586+\sqrt{3025971}}{3385}$& 10 & 1/5 & 3510 & 3795 & 3025971 & 466 \\
\hline
\end{tabular}

\end{table}

Finally, we present some experiments concerning square roots of integers in $\mathbb Q_5$. 

In Table \ref{tab:sq-per}, we collect some results for $\sqrt{\Delta} \in \mathbb Q_5$, with $0 < \Delta \leq 200$, for which a period was detected within the first $10000$ partial quotients. In particular, we report the length of the period, the number of complete quotients with negative norm within the period, the maximum values attained by $D_n$'s, maximum absolute difference between two consecutive values of $D_n$'s, the numbers of partial quotients whose Euclidean absolute value is, respectively, less than $1$, between $1$ and $2$, greater than $2$.  

\begin{table}[H]
\caption{$\sqrt{\Delta} \in \mathbb Q_5$ with periodic expansion}
\label{tab:sq-per}
\centering
\begin{tabular}{c c c c c c c c}
\hline
$\Delta$ & $\lambda(\sqrt{\Delta})$ & $N(\xi_n)<0$ & $\max D_n$ & $\max|D_n-D_{n-1}|$ & $|a_n| < 1$ & $1<|a_n|<2$ & $|a_n|>2$ \\
\hline
 6 &	4 & 0 & 50  & 45 & 1 & 3 & 0 \\
11 & 24 & 0 & 875 & 645 & 8 & 9 & 7 \\
14 & 6 & 0 & 125 & 70 & 0 & 6 & 0 \\
21 & 6 & 0 & 75 & 55 & 3 & 3 & 0 \\
24 & 2 & 0 & 20 & 5 & 0 & 1 & 1 \\
34 & 6 & 0 & 50 & 35 & 4 & 1 & 1 \\
54 & 10 & 0 & 230 & 185 & 5 & 5 & 0 \\
69 & 10 & 0 & 125 & 65 & 6 &4 & 0 \\
74 & 60 & 0 & 28975 & 24425 & 33 & 21 & 6 \\
76 & 26 & 2 & 2635 & 2275 & 12 & 10 & 4 \\
94 & 12 & 0 & 250 & 235 & 5 & 6 & 1 \\
99 & 8 & 0 & 450 & 355 & 5 & 3 & 0 \\
104	& 6 & 0 & 100 & 40 & 0 & 3 & 3 \\
111 & 10 & 0 & 435 & 325 & 8 & 2 & 0 \\
119 & 14 & 0 & 815 & 760 & 2 & 10 & 2 \\
\hline
\end{tabular}
\end{table}

In Table \ref{tab:sq-noper}, we report $\sqrt{\Delta} \in \mathbb Q_5$, with $0 < \Delta \leq 200$, for which no period was detected within the first $10000$ partial quotients. In particular, for $n \geq 2$, we report the number of complete quotients with negative norm, the maximum values attained by $D_n$, the maximum absolute difference between two consecutive values of $D_n$, the numbers of partial quotients whose Euclidean absolute value is, respectively, less than $1$, between $1$ and $2$, greater than $2$.
The maxima of $D_n$ and $|D_n - D_{n-1}|$ are reported on a logarithmic scale, using $\log_{10}$, since these quantities become extremely large. This representation also makes their order of magnitude, and hence approximately the number of their decimal digits, immediately evident.

\begin{longtable}{c c c c c c c}
\caption{$\sqrt{\Delta}\in\mathbb Q_5$ for which no periodicity was
detected within the first $10{,}000$ partial quotients}
\label{tab:sq-noper}\\
\hline
$\Delta$
& $N(\xi_n)<0$
& $\log_{10}(\max D_n)$
& $\log_{10}(\max |D_n-D_{n-1}|)$
& $|a_n|<1$
& $1<|a_n|<2$
& $|a_n|>2$ \\
\hline
\endfirsthead

\multicolumn{7}{c}{\tablename\ \thetable\ -- continued} \\
\hline
$\Delta$
& $N(\xi_n)<0$
& $\log_{10}(\max D_n)$
& $\log_{10}(\max |D_n-D_{n-1}|)$
& $|a_n|<1$
& $1<|a_n|<2$
& $|a_n|>2$ \\
\hline
\endhead

\hline
\endfoot
19 & 0 & 2065.5 & 2065.5 & 4010 & 3947	& 2042 \\
26 & 0 & 2189.7 & 2189.7 & 4003	& 3964	& 2032 \\
29 & 0 & 2083.4 & 2083.2 & 4001	& 4097	& 1901 \\
31 & 0 & 2160.5 & 2160.4 & 4061	& 3933	& 2005 \\
39 & 0 & 2087.9 & 2087.7 & 3996	& 3934	& 2069 \\
41 & 0 & 2145.5 & 2145.5 & 4024	& 4065	& 1910 \\
44 & 0 & 2120.3	& 2120.1 & 4043	& 3981	& 1975 \\
46 & 0 & 2061.8	& 2061.5 & 4116	& 3952	& 1931 \\
51 & 0 & 2143.4	& 2143.4 & 3995	& 3997	& 2007 \\
56 & 0 & 2217.7	& 2217.4 & 4010	& 3952	& 2037 \\
59 & 0 & 2128.7	& 2128.7 & 3949	& 4035	& 2015 \\
61 & 0 & 2121.7	& 2121.5 & 4051	& 3934	& 2014 \\
66 & 0 & 2124.7	& 2124.7 & 4031	& 3985	& 1983 \\
71 & 0 & 2143.9	& 2143.9 & 3951	& 4049	& 1999 \\
79 & 0 & 2158.6	& 2158.6 & 4026	& 3998 & 1975 \\
84 & 0 & 2151.6	& 2151.6 & 3999	& 4004	& 1996 \\
86 & 0 & 2174.8	& 2174.7 & 3988	& 3971	& 2040 \\
89 & 0 & 2223.7	& 2223.6 & 3967	& 3978 & 2054 \\
91 & 0 & 2067.8	& 2067.6 & 4049	& 3951 & 1999 \\
96 & 0 & 2216.0	& 2215.9 & 4023	& 4009 & 1967 \\
101 & 0 & 2178.2 & 2178.2 & 4014 & 3923 & 2062 \\
106 & 0 & 2219.0 & 2218.8 & 3909 & 4169 & 1921 \\
109 & 0 & 2163.8 & 2163.8 & 4026 & 3956 & 2017 \\
114 & 0 & 2151.8 & 2151.7 & 3990 & 3986 & 2023 \\
116 & 0 & 2118.4 & 2118.3 & 4081 & 3957 & 1961 \\
124 & 0 & 2190.7 & 2190.7 & 3967 & 4011 & 2021 \\
126 & 0 & 2127.6 & 2127.3 & 3982 & 4011 & 2006 \\
129 & 0 & 2082.9 & 2082.8 & 4044 & 4006 & 1949 \\
131 & 0 & 2109.1 & 2109.0 & 4026 & 3971 & 2002 \\
134 & 0 & 2193.1 & 2192.9 & 4013 & 4020 & 1966 \\
136 & 0 & 2169.4 & 2169.4 & 3997 & 4051 & 1951 \\
139 & 0 & 2191.4 & 2191.4 & 3973 & 4002 & 2024 \\
141 & 0 & 2164.2 & 2164.0 & 3946 & 4004 & 2049 \\
146 & 0 & 2192.8 & 2192.6 & 3937 & 4066 & 1996 \\
149 & 0 & 2160.0 & 2159.8 & 3996 & 4044 & 1959 \\
150 & 0 & 2144.7 & 2144.5 & 4018 & 3936 & 2045 \\
151 & 0 & 2230.5 & 2230.3 & 3956 & 4050 & 1993 \\
154 & 0 & 2225.1 & 2225.1 & 3930 & 4017 & 2052 \\
156 & 0 & 2139.2 & 2139.0 & 4004 & 4005 & 1990 \\
159 & 0 & 2226.7 & 2226.7 & 3864 & 4024 & 2111 \\
161 & 0 & 2184.8 & 2184.8 & 3943 & 4072 & 1984 \\
164 & 0 & 2127.9 & 2127.9 & 3984 & 4044 & 1971 \\
166 & 0 & 2091.6 & 2091.5 & 4001 & 3987 & 2011 \\
171 & 0 & 2100.3 & 2100.3 & 4053 & 3979 & 1967 \\
174 & 0 & 2189.8 & 2189.8 & 3950 & 4037 & 2012 \\
176 & 0 & 2119.6 & 2119.6 & 4044 & 3998 & 1957 \\
179 & 0 & 2157.3 & 2157.3 & 3976 & 4023 & 2000 \\
181 & 1 & 2135.6 & 2135.5 & 3951 & 4016 & 2032 \\
184 & 0 & 2167.1 & 2166.9 & 3912 & 4005 & 2082 \\
186 & 0 & 2194.3 & 2194.2 & 3975 & 4041 & 1983 \\
189 & 0 & 2102.6 & 2102.4 & 4085 & 3976 & 1938 \\
191 & 0 & 2136.7 & 2136.7 & 3990 & 4045 & 1964 \\
194	& 0 & 2136.3 & 2136.3 & 3937 & 4040 & 2022\\
199 & 0	& 2103.2 & 2103.2 & 4088 & 3987 & 1924\\
\hline
\end{longtable}

From the previous tables, positive norms for complete quotients appear to be the typical behaviour for square roots of integers, both in the periodic and in the apparently non-periodic cases. The main numerical distinction seems instead to concern the size of $D_n$ and $|D_n - D_{n-1}|$. Indeed, in the periodic examples these quantities are necessarily bounded, whereas in the cases for which no period was detected their running maxima grow very rapidly (consistently with Corollary \ref{cor:diff}). These experiments therefore provide further evidence that an analogue of Lagrange's theorem for Browkin continued fractions may fail. In particular, if the square roots listed in Table \ref{tab:sq-noper}
are periodic, their periods would have to be considerably longer than those observed in the periodic examples considered above (and compared to $\Delta$).
Moreover, the rapid growth of the quantities $D_n$ appears to make a return to a previously attained complete quotient increasingly unlikely. 

We can observe that Proposition \ref{prop:Npos-a1} applies to some periodic square roots of integers. For instance, consider $\sqrt{14}\in\mathbb Q_5$. Its Browkin continued fraction is
\[
\left[2, -\frac{3}{5} \overline{
-\frac95,-\frac65,\frac{166}{125},
-\frac65,-\frac95,-\frac85
}
\right],
\]
and all the complete quotients in the period have positive norm. The corresponding values of $D_n$ are
\[
5,\quad 55,\quad 125,\quad 55,\quad 5,\quad 10.
\]
Thus, the maximum is attained at an index $m$ for which
\[
D_m=125,\quad D_{m-1}=D_{m+1}=55, \quad |a_m|=\frac{166}{125}.
\]
Therefore, using the notation of Proposition \ref{prop:Npos-a1} we have
\[
K = \frac{1}{4}\left(\frac{166}{125}\right)^2
-\frac{55+55}{2\cdot125} =
\frac{14}{15625},
\]
and the bound gives
\[
D_m \leq \sqrt{\frac{14}{14/15625}} = 125.
\]
Hence, in this case the bound is attained exactly.
A similar phenomenon occurs for $\sqrt{104}\in\mathbb Q_5$.

The latter example also illustrates the advantage of the local
formulation of the proposition. Indeed, the partial quotients in the
period of $\sqrt{104}$ belong to both regimes
\[
1<|a_n|<2
\qquad\text{and}\qquad
|a_n|>2,
\]
so that none of the previous results requiring a uniform condition
on all the partial quotients applies directly. Nevertheless, the
behaviour at the index where $D_n$ attains its maximum is sufficient
to obtain an exact bound.

A further interesting feature emerges from the growth of the values $D_n$. In the apparently non-periodic cases, $\log_{10}(\max D_n)$ appears to grow approximately linearly with $n$. More specifically, define 
$$\gamma(n) := \frac{\log_{10}(\max_{j \leq n} D_j)}{n}.$$ 
We observed that this ratio is usually between $0.2$ and $0.3$. For instance, for $\sqrt{26}$, we obtain
\[ \gamma(100) = 0.28, \quad \gamma(500) = 0.24, \quad \gamma(1000) = 0.22, \quad \gamma(5000) = 0.22, \quad \gamma(10000) = 0.22 \]
while for $\sqrt{19}$ we have
\[ \gamma(100) = 0.25, \quad \gamma(500) = 0.21, \quad \gamma(1000) = 0.20, \quad \gamma(5000) = 0.21, \quad \gamma(10000) = 0.21 \]
More generally, for all the apparently non-periodic cases tested with $0<\Delta\leq 200$ we observed
\[ 0.202 \leq \gamma(5000) \leq 0.229, \quad 0.206 \leq \gamma(10000) \leq 0.223. \]
This remarkable stability suggests that, in these examples, the running maximum of $D_n$ may exhibit an approximately exponential growth with respect to $n$.

In conclusion, our results provide new restrictions on the possible behaviour of quadratic irrationals with non-periodic Browkin continued fraction expansion. In particular, any possible counterexample to an analogue of Lagrange's theorem must eventually have complete quotients of positive norm and must avoid several boundedness conditions established in the previous section. Indeed, for many square roots of integers for which no period was detected, the quantities $D_n$ and $|D_n-D_{n-1}|$ exhibit a rapid growth.
These observations leave several natural questions open. In particular, it would be interesting to obtain stronger periodicity criteria in the positive-norm case, especially when the partial quotients have norm less than 1. A deeper analysis of this behaviour may ultimately provide either new periodicity results or suitable candidates for disproving an analogue of Lagrange's theorem for Browkin continued fractions.
Moreover, it would be interesting to understand whether square roots of integers have eventually positive norms of their complete quotients, a behaviour that appears to occur in almost all the examples considered in our experiments.

\end{document}